\documentclass[11pt]{article}

\usepackage{amsmath,amssymb,amscd,amsthm}
\usepackage{ascmac}
\usepackage[all]{xy}
\usepackage{enumerate}
\usepackage{cancel}
\usepackage[dvips]{graphicx}
\usepackage{authblk}
\usepackage{comment}
\usepackage{bm}
\usepackage{mathrsfs}
\usepackage{array,arydshln}
\usepackage{geometry}
\usepackage{url}
\usepackage{ulem}
\theoremstyle{plain}
\newtheorem{thm}{Theorem}[section]
\newtheorem{pro}[thm]{Proposition}
\newtheorem{`thm'}[thm]{``Theorem''}
\newtheorem{cor}[thm]{Corollary}
\newtheorem{lem}[thm]{Lemma}

\newtheorem{dfn-thm}[thm]{Definition-Theorem}
\newtheorem{dfn-pro}[thm]{Definition-Proposition}
\newtheorem*{mainthm}{Main-Theorem}
\theoremstyle{definition}
\newtheorem{prob}[thm]{Problem}
\newtheorem{dfn}[thm]{Definition}

\newtheorem{asm}[thm]{Assumption}

\theoremstyle{remark}
\newtheorem{rmk}[thm]{Remark}
\newtheorem{rmks}[thm]{Remarks}
\newtheorem{ex}[thm]{Example}

\newcommand{\bb}[1]{\mathbb{#1}}
\newcommand{\fra}[1]{\mathfrak{#1}}
\newcommand{\ca}[1]{\mathcal{#1}}
\newcommand{\scr}[1]{\mathscr{#1}}
\newcommand{\Cl}{\mathit{Cliff}}

\newcommand{\inpr}[3]{\left(#1\middle|#2\right)_{#3}}

\newcommand{\pos}{{\rm pos}}

\newcommand{\td}{{\rm td}}
\newcommand{\ind}{{\rm ind}}

\newcommand{\id}{{\rm id}}
\newcommand{\PD}{{\rm PD}}
\newcommand{\loc}{{\rm loc}}
\newcommand{\fgt}{{\rm fgt}}

\newcommand{\New}{{\rm New}}

\newcommand{\Hom}{{\rm Hom}}
\newcommand{\End}{{\rm End}}

\newcommand{\midd}{\,\middle|\,}

\newcommand{\rank}{{\rm rank}}

\newcommand{\bra}[1]{\left(#1\right)}
\newcommand{\bbra}[1]{\left\{#1\right\}}
\newcommand{\bbbra}[1]{\left[#1\right]}

\newcommand{\ud}[1]{\underline{#1}}

\newcommand{\ev}{{\rm ev}}
\newcommand{\LC}{{\rm LC}}

\newcommand{\vep}{\varepsilon}
\newcommand{\grotimes}{\widehat{\otimes}}
\newcommand{\groplus}{\widehat{\oplus}}

\begin{document}
\title{An Index Theorem for Loop Spaces}
\author{Doman Takata \\
Niigata University}

\date{\today}

\maketitle
\begin{abstract}
We formulate and prove an index theorem for loop spaces of compact manifolds in the framework of $KK$-theory.
It is a strong candidate for the noncommutative geometrical definition (or the analytic counterpart)
of the Witten genus.
In order to find out an ``appropriate form'' of the index theorem to formulate a loop space version, we formulate and prove an equivariant index theorem for non-compact manifolds equipped with $S^1$-actions with compact fixed-point sets.
In order to formulate it, we use a ring of formal power series.\\\\
{\it Mathematics Subject Classification $(2010)$.} 
19K56 (Index theory); 
19K35 (Kasparov theory (KK-theory)), 
19L47 (Equivariant K-theory),
58B20 (Riemannian, Finsler and other geometric structures on infinite-dimensional manifolds),
58B34 (Noncommutative geometry).\\
{\it Key words}: index theorem, fixed-point formula, $KK$-theory, $\ca{R}KK$-theory, representable $K$-theory, loop space, $C^*$-algebras of Hilbert manifolds, Witten genus.
\end{abstract}

\setcounter{tocdepth}{2}
\tableofcontents

\section{Introduction}

\subsection*{Main theorem}

The Atiyah-Singer index theorem states that the analytic index of the Dirac operator on a closed manifold is determined by topological data \cite{ASi1,ASi2}.
When a manifold $M$ is equipped with a group action of a compact Lie group $G$, the index is an element of the representation ring of $G$ (called the equivariant index), and it is determined by data on the fixed-point set $M^G$ \cite{ASe1}.
Roughly speaking,
$$\text{analytic equivariant index}=\int_{M^G}(\text{topological data}).$$

Although the left hand side makes sense only for closed manifolds, the right hand side makes sense in much more general situations.
Witten defined an ``index of the Dirac operator on the free loop space on a compact manifold'' in \cite{Wit}.
This ``index'' is now called the {\bf Witten genus}.

The aim of the present paper is to formulate and prove an index theorem for loop spaces in the framework of $KK$-theory.
The main theorem is the following.
The notions appearing there will be explained soon.

\begin{mainthm}
For a compact $K$-oriented manifold $M$, we can define a homomorphism
$$\widetilde{\ind_{S^1}^\pos}:KK_{S^1}(\ca{A}(LM),\ca{S}_\vep)\to KK_{S^1}(\ca{S}_\vep,\ca{S}_\vep)_\pos$$
and it has a fixed-point formula.
\end{mainthm}

Let us briefly explain several symbols and notations: $LM$ is the free loop space of $M$; $\ca{A}(LM)$ is a $C^*$-algebra substituting for the ``function algebra of $LM$''; $KK_{S^1}(A,B)$ is the $S^1$-equivariant $KK$-theory, which coincides with the the set of homotopy classes of Dirac operators when $A$ is the function algebra of a closed manifold and $B=\bb{C}$; $\ca{S}_\vep$ is  a certain $C^*$-algebra defined in Definition \ref{Cstar alg S};
$KK_{S^1}(\ca{S}_\vep,\ca{S}_\vep)_\pos$ is a kind of representation ring of $S^1$ which is inspired by the concept of positive energy representations of loop groups \cite{PS}.
In short, we will formulate and prove an {\bf index theorem for loop spaces!}

Although we have not verified that the above homomorphism realizes the Witten genus, the homomorphism is probably quite useful to realize it.

\subsection*{Equivariant index theorem and $KK$-theory}

Let us move on to expositions of previous researches related to the present paper.
We begin with a $KK$-theoretical formulation of the equivariant index theorem.
One can conceptually understand equivariant index theory using $KK$-theory and $\ca{R}KK$-theory \cite{Kas88,Kas15,Bla}.
Roughly speaking, analytic data and topological data corresponds to each other, not only at the level of index, but also at the level of groups in which these data live.



$KK$-theory is a bivariant functor from the category of $C^*$-algebras to the category of abelian groups.
For a pair of $C^*$-algebras $(A,B)$, the $KK$-group is denoted by $KK(A,B)$, and it is contravariant in $A$ and covariant in $B$.
It has an equivariant version. 
For special cases, $KK$-theory has a topological interpretation. If $X$ is a locally compact Hausdorff space, $KK(\bb{C},C(X))$ is isomorphic to Atiyah's $K$-group.


%

$\ca{R}KK$-theory is a joint generalization of $KK$-theory and Segal's $RK$-theory \cite{Kas88,Kas15}.
For a locally compact Hausdorff space $X$, a $C^*$-algebra $A$ is said to be a {$C_0(X)$-algebra} if it is equipped with a $*$-homomorphism $C_0(X)\to \scr{Z}(\ca{M}(A))$, where $\ca{M}(A)$ is the multiplier algebra of $A$ and $\scr{Z}(\ca{M}(A))$ is its center.
For a locally compact Hausdorff space $X$ and a pair of $C_0(X)$-algebras $(A,B)$, we can define an abelian group $\ca{R}KK(X;A,B)$.
For example, $\ca{R}KK(X;C_0(X),C_0(X))$ is isomorphic to Segal's representable $K$-group of $X$, $RK^0(X)$.
More generally, $\ca{R}KK$-theory admits the $K$-theory version of the ``compact vertical cohomology''.
For a vector bundle $E$ over $X$, it is given by
$K_{cv}^0(E)=\ca{R}KK(X;C_0(X),C_0(E))$.
In this sense, $\ca{R}KK$-theory is ``more topological'' than $KK$-theory.
It also has an equivariant version.

$KK$-theory and $\ca{R}KK$-theory provide a nice framework to deal with index theory.
When $X$ is a complete Riemannian manifold equipped with an isometric action of a locally compact group $G$, a $G$-equivariant Dirac operator acting on a Clifford bundle $W$ over $X$ determines an element $[D]\in KK_G(C_0(X),\bb{C})$.
We call it the {\bf index element of $D$}.
Moreover, a fiberwise linear map $\sigma_D:\varpi^*W\to \varpi^*W$ over $TX$ is defined, where $\varpi:TX\to X$ is the natural projection.
Since $\sigma_D$ is invertible outside the zero section, it determines an element $[\sigma_D]\in \ca{R}KK_G(X;C_0(X),C_0(TX))$.
We call it the {\bf symbol element of $D$}.
By Theorem 4.1 and 4.3 in \cite{Kas15}, $KK_G(C_0(X),\bb{C})$ and $\ca{R}KK_G(X;C_0(X),C_0(TX))$ are isomorphic, and the index element of $D$ corresponds to the symbol element of $D$ under this correspondence.
This correspondence is called the {\bf $K$-theoretical Poincar\'e duality}.

Using this framework, Hochs and Wang generalized the fixed-point formula in \cite{HW}. 
For a complete Riemannian manifold $X$ equipped with a torus action so that the fixed-point set is compact, they defined an analytic index homomorphism
$$\ind_g:KK_G(C_0(X),\bb{C})\to R(G)_g,$$
where $R(G)_g$ is the localization of the representation ring of $G$ at $g$ in the algebraic sense.
Then, they proved the fixed-point formula.
When $X$ is compact, the index equals to the Atiyah-Segal-Singer's index.
Our construction is very much inspired by this result.
We will review it in detail in order to compare it with our index in Section \ref{Comparison with the Hochs-Wang's index}.

There are other $KK$-theoretical formulations of equivariant index theory: for example \cite{LRS,HS, Kas15} and related papers.

\subsection*{Infinite-dimensional manifolds}

Although noncommutative geometry is a powerful tool to generalize index theory, it has a weak point.
In order to translate something topological into the language of $C^*$-algebras, we use the Gelfand-Naimark representation theorem: The functor ``taking the algebra consisting of continuous functions vanishing at infinity'' is a contravariant category equivalence between the category of locally compact Hausdorff spaces and the category of commutative $C^*$-algebras.
This result means that a function algebra in the ordinary sense of a non-locally compact space, for example an infinite-dimensional manifold, is the function algebra of a {\it different} locally compact Hausdorff space.
For example, a continuous function vanishing at infinity of an infinite-dimensional Hilbert space is zero.
Therefore, if we want to study an infinite-dimensional manifold using noncommutative geometry, at least, {\it a ``$C^*$-algebra substituting for a function algebra of a non-locally compact space'' must be noncommutative.}

Higson, Kasparov and Trout defined such a $C^*$-algebra for an infinite-dimensional Hilbert space \cite{HKT}.
It is an infinite-dimensional analogue of the ``graded suspension of the Clifford algebra-valued function algebra''.
We call it the ``$C^*$-algebra of a Hilbert space''.
It was used to study Baum-Connes conjecture for a-T-menable groups in \cite{HK}.
It was generalized to Hilbert bundles and it was used to formulate a Thom isomorphism for Hilbert bundles in \cite{Tro}.
We will give an alternative definition of the Thom homomorphism in the present paper.
Moreover, the construction was generalized to general Hilbert manifolds in \cite{DT}.
However, in the present paper, we will use an alternative generalization explained in the next paragraph.

The original construction of the $C^*$-algebra of a Hilbert space $H$ is the inductive limit of the graded suspension of the Clifford algebra-valued function algebra of finite-dimensional subspaces of $H$.
More geometrical definition of such $C^*$-algebras was given in \cite{GWY}.
Gong, Wu and Yu defined a ``$C^*$-algebra of a Hilbert-Hadamard space'' using the exponential mapping, where a Hilbert-Hadamard space is a complete geodesic CAT(0) metric space all of whose tangent cones are isometrically embeddable into Hilbert spaces.
Then, Yu generalized it to Hilbert manifolds whose all injectivity radii are bounded below in \cite{Yu}.
This construction is one of the leading actors in the present paper.

Regarding index theory of infinite-dimensional manifolds, we need to refer to Hamiltonian loop group spaces.
This concept was introduced in \cite{MW}.
For example, the moduli space of flat connections on a Riemann surface with boundary is a Hamiltonian loop group space.
There is a one-to-one correspondence between Hamiltonian loop group spaces and quasi-Hamiltonian spaces \cite{AMM}.
In \cite{LMS,LS,Song}, several topics related to index theory are studied.
In particular, Song constructed an ``analytic index for a Hamiltonian loop group space'' by using a Hilbert bundle over the corresponding quasi-Hamiltonian space in \cite{Song}.
Inspired by it, we studied index theory of Hamiltonian loop group space for the circle group in \cite{T1,T2,Thesis,T4,T5}.
A strong point of our result compared to others is that the construction is $KK$-theoretical.
These studies play an important role in the present paper as we explain in the next paragraph.

\subsection*{Infinite-dimensional $K$-theoretical Poincar\'e duality}

As we have pointed out, $\ca{R}KK$-theory has a topological flavor. Thus, it is possible to generalize 
$\ca{R}KK$-theory to non-locally setting by reformulating the theory by using fields of $C^*$-algebras, Hilbert modules, homomorphisms and operators instead of single objects  \cite{T5}.
Detailed properties will be studied in \cite{NT}.

By using non-locally compact equivariant $\ca{R}KK$-theory and Yu's $C^*$-algebras of Hilbert manifolds, we formulated an infinite-dimensional version of the $K$-theoretical Poincar\'e duality homomorphism in \cite{T5}.
The Bott periodicity map in \cite{HKT} was also important.
The infinite-dimensional version of the $K$-theoretical Poincar\'e duality homomorphism plays a central role in the proof of the fixed-point formula for loop spaces.

\subsection*{Structure of paper}

In section 2, we will review several necessary operations on $\ca{R}KK$-theory, and we will prepare several basic $KK$-elements and $\ca{R}KK$-elements.
We will define the Bott element, the Dirac element, the Thom element, its inverse (fiberwise Dirac element) and the local Bott element. 
We will also give a $KK$-equivalence between $C_0(X)$ and $C_0(X,\Cl_+(TX))$ for $K$-orientable $X$, where $\Cl_+(TX)$ is the Clifford algebra bundle of $TX$.
The Bott element and the Dirac element for a Euclidean space are essential.
This is because the Thom element, its inverse and the local Bott element are in some sense family versions of the Bott elements or the Dirac elements.
We will also briefly review non-locally compact groupoid equivariant $KK$-theory. For detailed expositions, see \cite{T5,NT}.

In Section 3, we will construct an $S^1$-equivariant index for non-compact manifolds equipped with isometric $S^1$-actions with compact fixed-point sets.
We will call it the {\bf localized index}.
We will use a kind of the ring of formal power series instead of localization of commutative ring which was used in \cite{ASe1,HW}.
Then, we will deduce the cohomological formula for this index by using the $K$-theoretical Poincar\'e duality. 
We will rewrite it for $K$-oriented manifolds.
In this case, the localized index of a Dirac operator $D$ can be computed by the integration of a characteristic class of $D$ over the fixed-point set.
Then, in order to generalize it to loop spaces, we will reformulate it for $K$-oriented manifolds using only appropriate $C^*$-algebras.
We will also compare our index and \cite{HW}'s index.

In Section 4, we will construct a loop space version of the localized index and we will prove the fixed-point formula, by constructing a loop space version of each step of the construction of the localized index.
This section contains a functorial study on $C^*$-algebras of Hilbert manifolds.
In particular, we will give an alternative infinite-rank Thom homomorphism for some special cases.
We will show several remained problems.

In Appendix, we will prove that when we apply the construction of the localized index for a closed manifold, we obtain the classical index.
In this sense, our localized index is appropriate.

\subsection*{Table of notations}

\begin{itemize}
\setlength{\parskip}{0cm} 
  \setlength{\itemsep}{0cm} 
%

\item All irreducible representations of $S^1$ are $1$-dimensional and classified by weight.
The weight $k$-representation space is denoted by $\bb{C}_k$.
The corresponding element in $R(S^1)$ is denoted by $q^k$.


\item For a $\bb{Z}_2$-graded vector space $V=V_0\widehat{\oplus}V_1$, we denote the grading by $\partial$ and the graded homomorphism by $\epsilon$, that is to say, for $v\in V_i$ ($i=0$ or  $1$), $\partial v:=i$ and $\epsilon (v)=(-1)^iv=
(-1)^{\partial v}v$.
An element of $V_0\cup V_1$ is said to be homogeneous.
\item The symbol $\grotimes$ means the graded tensor product of Hilbert spaces, $C^*$-algebras or Hilbert modules.
The multiplication on the graded tensor product of graded $C^*$-algebras $A\grotimes B$ is defined by $(a_1\grotimes b_1)\cdot (a_2\grotimes b_2)
:=(-1)^{\partial(a_2)\partial(b_1)}(a_1a_2)\grotimes (b_1b_2)$ for homogeneous $b_1$ and $a_2$.
We often use this symbol even if one of the gradings is trivial.


\item For a $C^*$-algebra $B$ and a Hilbert $B$-module $E$, the set of adjointable bounded operators is denoted  by $\bb{L}_B(E)$ or simply by $\bb{L}(E)$. 



\item A $*$-homomorphism from a $C^*$-algebra to $\bb{L}_B(E)$ for a $C^*$-algebra $B$ and a Hilbert $B$-module $E$, is almost always denoted by $\pi$.
We often use the same symbol to denote other $*$-homomorphisms of this type.


\item For a Euclidean space $V$, its Clifford algebras are denoted by $\Cl_+(V)=T(V)/(v^2\sim \|v\|^2)$ and $\Cl_-(V)=T(V)/(v^2\sim -\|v\|^2)$. 
We will use both of them.
This construction also works for Euclidean vector bundle.
For a Euclidean vector bundle $E$ over $X$, we can define a $C^*$-algebra bundle $\Cl_\pm(E)$ over $X$ in an obvious way.

\item In the present paper, a ``Spinor of $V$'' means a $\bb{Z}_2$-graded irreducible representation of $\Cl_-(V)$.
More concretely, a Spinor of $V$ is a Hermitian vector space $S$ equipped with a liner map $c:V\to \End(S)$ so that $c(v)$ is skew Hermitian and $c(v)^2=-\|v\|^2\id$.
We call an irreducible representation of $\Cl_+(V)$ a dual Spinor.


\item For a locally compact Hausdorff space $X$, $C_0(X)$ is the $C^*$-algebra consisting of $\bb{C}$-valued continuous functions vanishing at infinity. More generally, for a bundle of $C^*$-algebras $A=\{A_x\}_{x\in X}$ over $X$, $C_0(X,A)$ is the $C^*$-algebra consisting continuous sections of $A$ vanishing at infinity.

\item For a topological space $X$ and a Euclidean vector bundle $E$ over $X$, we denote $C_0(X,\Cl_+(E))$ by $Cl_E(X)$. When $X$ is a Riemannian manifold, we denote $Cl_{TX}(X)$ by $Cl_\tau(X)$.


\item The projection of a fiber bundle is almost always denoted by $\varpi$.
We often use the same symbol to denote other fiber bundles.

\end{itemize}

\section{Preliminaries}

\subsection{Basic $KK$-elements and index theorem}

\subsubsection{$\ca{R}KK$-theory}\label{RKK}

In this subsection, we prepare necessary $KK$-elements and $\ca{R}KK$-elements from \cite{Kas15}.

Kasparov's equivariant $\ca{R}KK$-theory is a joint generalization of Kasparov's $KK$-theory and Segal's $RK$-theory \cite{Kas88}.
It is an invariant for the following data: A locally compact space $X$, a locally compact Hausdorff group $G$ acting on $X$, and a pair of $C_0(X)$-$G$-$C^*$-algebras $A$ and $B$. It is denoted by $\ca{R}KK_G(X;A,B)$.
It differs from $KK_G(A,B)$ only in the following additional requirement; if $(E,\pi,F)$ is a $G$-equivariant Kasparov $(A,B)$-module, then for any $f\in C_0(X)$, $a\in A$, $b\in B$ and $e\in E$, one has $\pi(f\cdot a)(eb)=\pi(a)(e(f\cdot b))$.

The tensor product of two $C_0(X)$-$C^*$-algebras $B_1$ and $B_2$ has at least two different $C_0(X)$-$C^*$-algebra structure. If the $C_0(X)$-$C^*$-algebra structure is given by $f\cdot (b_1\otimes b_2)=(f\cdot b_1)\otimes b_2$, we denote it by $\uwave{B_1}\otimes B_2$;
if it is given by $f\cdot (b_1\otimes b_2)=b_1\otimes (f\cdot b_2)$, we denote it by $B_1\otimes \uwave{B_2}$.

A $C_0(X)$-$G$-$C^*$-algebra can be regarded as a ``$G$-equivariant family of $C^*$-algebras'' \cite{Nil,Blan}.
The corresponding family of $C^*$-algebras of a $C_0(X)$-algebra $A$ is denoted by $\{A_x\}_{x\in X}$.
With a similar technique, roughly speaking, an $\ca{R}KK_G(X;A,B)$-cycle can be described as a ``$G$-equivariant upper semi-continuous family of Kasparov $(A_x,B_x)$-modules''.
The corresponding family of Kasparov modules of an $\ca{R}KK_G(X;A,B)$-cycle $(E,\pi,T)$ is denoted by $\{(E_x,\pi_x,T_x)\}_{x\in X}$.

$\ca{R}KK$-theory is contravariant in $X$.
We need only the following functoriality.

\begin{dfn}
Let $X$ be a $G$-space and let $Y$ be an invariant subspace.
The natural inclusion is denoted by $\iota:Y\hookrightarrow X$.
Let $A$ and $B$ be $G$-$C_0(X)$-$C^*$-algebras.
In the family picture of $\ca{R}KK$-theory,
$\iota^*:\ca{R}KK_G(X;A,B)\to \ca{R}KK_G(Y;A|_Y,B|_Y)$ is defined by the restriction of everything:
$$\iota^*:\ca{R}KK_G(X;A,B)\ni
\{(E_x,\pi_x,F_x)\}_{x\in X}
\mapsto
\{(E_x,\pi_x,F_x)\}_{x\in Y}\in \ca{R}KK_G(Y;A|_Y,B|_Y).$$
\end{dfn}
\begin{rmks}
$(1)$ In the single module picture, $\iota^*$ is given by the tensor product with $C_0(Y)$ over $C_0(X)$. Note that $C_0(Y)$ is a $C_0(X)$-algebra by $\iota^*:C_0(X)\to B(Y)$, where $B(Y)$ is the $C^*$-algebra consisting of bounded continuous functions on $Y$. Note that $\iota^*$ does not take values in $C_0(Y)$ unless $\iota$ is proper.

$(2)$ $A|_Y$ is defined by the section algebra $C_0(Y,\{A_y\}_{y\in Y})$.
\end{rmks}
\begin{ex}
If $X$ is a manifold, $C_0(TX)|_Y=C_0(TX|_Y)$.
\end{ex}

\subsubsection{Bott element and Dirac element}\label{B and D}

For a Euclidean space $V$ equipped with an orthogonal linear action $G$, $Cl_\tau(V)$ and $\bb{C}$ are $KK_G$-equivalent by the following Bott element and the Dirac element: In the unbounded picture, they are defined by
$$[b_V]=\bbbra{\bra{Cl_\tau(V),\pi,C}}\in KK_G(\bb{C},Cl_\tau(V)),$$
$$[d_V]=\bbbra{\bra{L^2(V,\Cl_+(V)),\pi,\sum c^*(e_i)\frac{\partial }{\partial x_i}}}\in KK_G(Cl_\tau(V),\bb{C}),$$
where $\{e_i\}$ is an orthonormal base of $V$, $c(e_i)$ is the Clifford multiplication, $c^*(e_i)v:=(-1)^{{\rm deg}(v)}ve_i$, and $C=\sum x_ic(e_i)$ is so called the Clifford operator.
It is possible to consider parallel constructions for the open ball $B$ of radius $\vep$ centered at the origin of $V$.
For this case, it is often more convenient to use 
$$\bra{Cl_\tau(B),\pi,\vep^{-1}C}$$
as the representative of $[b_B]$.
The norm of $\vep^{-1}C$ is $1$ on the boundary of $B$, and hence $1-(\vep^{-1}C)^2$ is $Cl_\tau(B)$-compact.
Note that it is in the bounded picture.

Similarly, $C_0(V)$ and $\Cl_+(V)$ are $KK_G$-equivalent.
This equivalence is also useful.

The Bott periodicity map is realized by a $*$-homomorphism in the following sense.

\begin{dfn}\label{Cstar alg S}
$(1)$ We define a $\bb{Z}_2$-graded $C^*$-algebra $\ca{S}$ as follows: The underlining $C^*$-algebra is $C_0(\bb{R})$ and the $\bb{Z}_2$-grading is given by the homomorphism $\epsilon f(t):=f(-t)$.
It has an unbounded multiplier $X$ defined by $(Xf)(t):=tf(t)$.

$(2)$ Similarly, we define a subalgebra $\ca{S}_\vep=C_0(-\vep,\vep)$ and the {\it bounded} multiplier $X$ by the restriction of $X$.
\end{dfn}

\begin{lem}[{\cite[Section 2.3]{T5}}]
$\ca{S}$ and $\ca{S}_\vep$ are $KK$-equivalent to $\bb{C}^2$.
\end{lem}

\begin{pro}[{\cite{HKT}, \cite[Proposition 2.19]{T5}}]
Let $\beta:\ca{S}\to \ca{S}\grotimes Cl_\tau(V)$ be the $*$-homomorphism defined by
$$\beta(f):=f(X\grotimes 1+1\grotimes C).$$
Then, if $V$ is even-dimensional, $[\beta]=\sigma_{\ca{S}}([b_V])$.
\end{pro}


Family versions of $[b_V]$ and $[d_V]$ play a crucial role in Kasparov's index theory.
We will explain two aspects: the Thom isomorphism and the $K$-theoretic Poincar\'e duality.

\subsubsection{Thom isomorphism}

Let $E$ be a $G$-equivariant Euclidean vector bundle over a locally compact Hausdorff space $X$.
For $x\in X$, following the vector space cases, we define $c_x:E_x\to \End(\Cl_+(E_x))$ by $c_x(v)(e):=v\cdot e$ and
$c^*_x:E_x\to \End(\Cl_+(E_x))$ by $c^*_x(v)(e):=(-1)^{\partial e}e\cdot v$.
$c_x$ extends to $\Cl_+(E_x)$.
Then, two $C^*$-algebras $C_0(E)$ and $Cl_E(X):=C_0(X,\Cl_+(E))$ are $KK_G$-equivalent as follows.

\begin{dfn}[{\cite[2.5-2.7]{Kas15}}]\label{Thom for general}
$(1)$ $[\ca{B}_{E}]\in \ca{R}KK_G(X;Cl_E(X),C_0(E))$ is defined by
$$\bbra{\Bigl(
C_0(E_x)\otimes \Cl_+(E_x),c_x,iC_x
\Bigr)}_{x\in X},$$
where for a homogeneous element $e\in C_0(E_x)\otimes \Cl_+(E_x)$, we define $C_x(e)$ by $C_x(e)(v):=c_x^*(v)(e(v))=(-1)^{\partial e}e(v)\cdot v$.


$(2)$ $[d_E]\in \ca{R}KK_G(X;C_0(E),Cl_E(X))$ is defined by
$$\bbra{\bra{
L^2(E_x)\otimes \Cl_+(E_x),\pi,
\sum_{k}\frac{\partial}{\partial \xi_k}\otimes c^*_x(e_k)
}}_{x\in X},$$
where $\pi$ is the multiplication by $C_0(E_x)=C_0(E)|_x$ on $L^2(E_x)$.
\end{dfn}
\begin{rmk}
$[\ca{B}_E]$ is a family of Bott elements, and $[d_E]$ is a family of Dirac elements.
\end{rmk}

The following is an obvious generalization of \cite[Theorem 2.7]{Kas15}.

\begin{thm}\label{Thom isom E}
$[d_E]\grotimes_{Cl_E(X)}[\ca{B}_E]=\bm{1}_{C_0(E)}$ and 
$[\ca{B}_E]\grotimes_{C_0(E)}[d_E]=\bm{1}_{Cl_E(X)}$.
Consequently, $C_0(E)$ and $Cl_E(X)$ are $KK_G$-equivalent.
This equivalence is called the {\bf Thom isomorphism}.
\end{thm}

\subsubsection{Poincar\'e duality homomorphism}

By $KK$-theoretical Poincar\'e duality homomorphism, a $KK$-element represented by a Dirac operator is transformed into its symbol.
Kasparov realized such a homomorphism by using the Kasparov product with an $\ca{R}KK$-element.

\begin{dfn}[{\cite[Definition 2.3]{Kas15}}]\label{Def PD}
Let $X$ be a complete Riemannian manifold equipped with an isometric action of a locally compact group $G$ (in the present paper, 
we will deal with only $G=S^1$).
For simplicity, in the present paper, we suppose that there is a positive real number $\vep$ such that the injectivity radius at any $x\in X$ is greater than $2\vep$.

$(1)$ Let $U_x$ be the $\vep$-ball centered at $x$ in $X$.
We define $\Theta_x:U_x\to T_xX$ by
$$\Theta_x(y):=\log_x(y)=\text{``}\overrightarrow{xy}\text{''}\in T_xX,$$
where $\log_x:U_x\to T_xX$ is the local inverse of $\exp_x:T_xX\to X$. 
The {\bf local Bott element}\footnote{This $\ca{R}KK$-element is denoted by $[\Theta_{X,2}]$ in \cite{Kas15}.} $[\Theta_{X}]$ is defined by the element of $\ca{R}KK_G(X;C_0(X),C_0(X)\grotimes \uwave{Cl_\tau(X)} )$ represented by the family of Kasparov modules
$$\bbra{
\Bigl(
C_0(U_x)\otimes \Cl_+(T_xX),1_x,\vep^{-1}\Theta_x
\Bigr)}_{x\in X},$$
where $1_x$ denotes the homomorphism given by $z\mapsto z\id$, and $\Theta_x$ denotes the left multiplication by $\Theta_x$.


$(2)$ The homomorphism $\PD_X:KK_G(C_0(X),\bb{C})\to \ca{R}KK_G(X;C_0(X),Cl_\tau(X))$ is defined by 
$\PD_X([D]):=[\Theta_{X}]\grotimes_{C_0(X)}[D]$.\footnote{
Strictly speaking, this Kasparov product stands for
$[\Theta_{X}]\grotimes_{X,C_0(X)\grotimes \uwave{Cl_\tau(X)}}\bbra{\sigma_{X,Cl_\tau(X)}\bra{[D]}}.$}
We call the homomorphism $\PD_X$ the {\bf Poincar\'e duality homomorphism}.
We call the element $\PD_X([D])$ the {\bf Clifford symbol element}.

\end{dfn}

The symbol element of a Dirac operator is defined as follows.

\begin{dfn}
For a $G$-equivariant Dirac operator $D$ on a Clifford module bundle $W$ equipped with a Clifford multiplication $c_W:TX\to \End(W)$ satisfying that $c_W(v)=-\|v\|^2$, {\bf the symbol of $D$} is defined by, in the unbounded picture,
$$[\sigma_D]:=\bbbra{\bra{C_0(TX,\varpi^*W),\pi,ic_W}}\in
\ca{R}KK_G(X;C_0(X),C_0(TX)),$$
where $\varpi:TX\to X$ is the natural projection.
\end{dfn}

$[D]$ and $[\sigma_D]$ correspond to each other under the Poincar\'e duality and the Thom isomorphism.

\begin{thm}[{\cite[Theorem 4.3]{Kas15}}]\label{thm inv ind thm}
$[\sigma_D]=\PD_X([D])\grotimes [\ca{B}_{TX}]$.
\end{thm}

\subsubsection{For $K$-orientable bundles}

A Euclidean vector bundle is said to be {\bf $K$-orientable} if it is of even-rank and of $Spin^c$.
A $K$-orientable Euclidean vector bundle is said to be {\bf $K$-oriented} if a Spinor bundle is fixed.
We can define the concept ``equivariantly $K$-orientable'' and 
``equivariantly $K$-oriented'' in an obvious way.

For a $G$-equivariantly $K$-oriented vector bundle $E$ over $X$, $C_0(X)$, $C_0(E)$ and $Cl_E(X)$ are $KK_G$-equivalent.
Let $S_E$ be a Spinor bundle for $E$, that is to say, $S_E$ is a Hermitian vector bundle equipped with a linear map $\gamma:E\to \End(S_E)$ satisfying that $\gamma(v)^2=-\|v\|^2\id$ and $\gamma(v)^*=-\gamma(v)$.
Then, $S_E$ is automatically equipped with a right Hilbert $\Cl_+(E)$-module bundle structure thanks to $\Cl_-(E)\cong \End(S_E)$.
See \cite[Section 3.2]{T5} for details.
Thus, we can define two $\ca{R}KK_G$-elements
$$[S_E]=(C_0(X,S_E),\pi,0)\in \ca{R}KK_G(X;C_0(X),Cl_E(X)),$$
$$[S_E^*]=(C_0(X,S_E^*),\pi,0)\in \ca{R}KK_G(X;Cl_E(X),C_0(X))$$
and they give $\ca{R}KK_G$-equivalence between $C_0(X)$ and $Cl_E(X)$.

\begin{dfn}\label{KK equiv K ori}
For a $K$-oriented vector bundle $E$ over $X$, we define
$[\ca{B}_E^{Spin^c}]
:=[S_E]\grotimes [\ca{B}_E]
\in \ca{R}KK_G(X;C_0(X),C_0(E))$ and 
$[d_E^{Spin^c}]
:=[d_E]\grotimes [S_E^*]
\in \ca{R}KK_G(X;C_0(E),C_0(X))$.
\end{dfn}
\begin{rmks}
$(1)$ $[\ca{B}_E^{Spin^c}]$ and $[d_E^{Spin^c}]$ are mutually inverse.

$(2)$ A Riemannian manifold $X$ is said to be $K$-orientable if $TX$ is $K$-orientable, and  $X$ is said to be $K$-oriented if $TX$ is $K$-oriented.
For a $K$-orientable manifold $X$, $Cl_\tau(X)$ and $C_0(X)$ are $\ca{R}KK_G$-equivalent.

$(3)$ If a vector bundle $\varpi:E\to X$ has a complex structure, it has a Spinor bundle 
$S_E=\wedge^*{\varpi^*(E)}$ given by the complex exterior algebra bundle.
A Clifford multiplication of $v\in E$ is given by $v\wedge+(v\wedge)^*$.
It will appear in the construction of the index homomorphism of the present paper.
\end{rmks}



%

\subsection{$\ca{R}KK$-theory for non-locally compact groupoids}

Kasparov's equivariant $\ca{R}KK$-theory is a joint generalization of Kasparov's $KK$-theory and Segal's $RK$-theory.
It was extended by Le Gall as groupoid equivariant $KK$-theory in \cite{LG}.

$\ca{R}KK$-theory has a topological flavor.
For example, $\ca{R}KK(X;C_0(X),C_0(X))\cong RK(X)$, where the right hand side is Segal's representable $K$-theory.
Therefore, one can extend this invariant for much more general situations by using the family description of $\ca{R}KK$-theory: {\it We can define non-locally compact groupoid-equivariant $KK$-theory} \cite{T5,NT}.
Let $\scr{G}$ be a possibly non-locally compact groupoid.
We explain the necessary changes to define it from the locally compact cases.
For simplicity, we explain it for an action groupoid $\scr{G}=\ca{X}\rtimes \ca{G}$ for a possibly non-locally compact normal space $\ca{X}$ and a Hausdorff group $\ca{G}$  properly acting on $\ca{X}$.
\begin{itemize}
\item We need to replace $C_0(X)$-$G$-$C^*$-algebras $A$ and $B$ with ``$\ca{G}$-equivariant upper semi-continuous fields of $C^*$-algebras parameterized by $\ca{X}$'', $\scr{A}$ and $\scr{B}$.
\item We need to replace a Hilbert module $E$ with a ``$\ca{G}$-equivariant upper semi-continuous field of Hilbert modules parameterized by $\ca{X}$'', $\scr{E}$.
\item We need to replace a $*$-homomorphism $\pi:A\to \bb{L}_{B}(E)$ with a ``continuous family of $*$-homomorphisms'', $\pi=\{\pi_x:A_x\to \bb{L}_{B_x}(E_x)\}$.
\item We need to replace an adjointable operator $F\in \bb{L}_{B}(E)$ satisfying several conditions with a bounded lower semi-continuous family of adjointable operators $\{F_x\in \bb{L}_{B_x}(E_x)\}$.
\end{itemize}
It is covariant in $\scr{B}$, contravariant in $\scr{A}$ and $(\ca{X},\ca{G})$. It will be extensively studied in \cite{NT}.


\section{An $S^1$-equivariant index theorem for non-compact manifolds}

In this section, we define an $S^1$-equivariant index for non-compact manifolds. The construction is inspired by \cite{HW}.
In the present paper, we use a kind of ring of formal power series instead of localization in the sense of commutative rings.
Then, we will deduce the fixed-point theorem by the $K$-theoretical Poincar\'e duality homomorphism.
We will prove that there are no essential differences between our index and \cite{HW}'s index.
However, our construction has an advantage that we can apply the construction to loop spaces.

\subsection{Construction of $S^1$-equivariant index for non-compact manifolds}\label{Non-compact manifolds}

Let $X$ be a finite-dimensional complete Riemannian manifold equipped with an isometric action of $S^1$. 
We suppose that the fixed-point set $M:=X^{S^1}$ is compact.
We will prove an index theorem for this situation.
In order to construct an analytic index, we prepare several subsets of $X$.


Let $\nu(M )$ be the normal bundle of $M $ in $X$.
Then, we have a normal exponential mapping $\exp^\perp:\nu(M )\to X$, which is diffeomorphic on a neighborhood of the zero section. 
For $\delta>0$, let 
$\nu(M )_\delta:=\bbra{v\in \nu(M ) \midd \|v\|<\delta}$.
Let $U_\delta:=\exp^\perp(\nu(M )_\delta)$. 
Suppose that $\delta$ is less than the injectivity radius of $\exp^\perp$.
The inclusions are denoted by $j:M \hookrightarrow U_\delta$ and $k:U_\delta\hookrightarrow X$. 
The composition of them is denoted by $i=k\circ j:M \hookrightarrow X$.
Since $U_\delta$ is diffeomorphic to a subset of the normal bundle, we have a projection 
$\varpi:U_\delta \to M$.
The setting is summarized as follows:
$$\xymatrix{
M  \ar@{^{(}->}^-{j}[r] \ar@/^28pt/_{i}[rr] & U_\delta \ar@{^{(}->}^-{k}[r] & X. \\
& \nu(M )_\delta \ar@{->>}^-{\varpi}[lu] \ar_-{\exp^{\perp},\cong}[u] &}$$

An index homomorphism $\ind_X$ of $X$ should satisfy that
$\ind_X(i_*[D])=\ind_{M }([D])$ for $[D]\in K_{S^1}(C(M ),\bb{C})$.
Thus, we hope to define 
$$\ind_X:=\text{``}\ind_{M }\circ i_*^{-1}\text{''}.$$
We will prove that $i_*$ is invertible after a certain algebraic operation.

In order to study $i_*$, we focus on the following commutative diagram:
$$\xymatrix{
& KK_{S^1}(C_0(U_\delta),\bb{C}) & \\
KK_{S^1}(C(M ),\bb{C})
 \ar^{j_*}[ru] 
 \ar_{i_*}[rr]   && 
KK_{S^1}(C_0(X),\bb{C}). \ar_{k^*}[lu] }$$
It is deduced from the following commutative diagram:
$$\xymatrix{
& C_0(U_\delta) \ar_{j^*}[ld] \ar^{k_*}[rd] & \\
C(M ) && C(X)\ar^{i^*}[ll] .}$$

If they make sense and they are isomorphic, $(i_*)^{-1}$ should be $(j_*)^{-1}\circ k^*$.

At least, $KK_{S^1}(C_0(U_\delta),\bb{C})$ and $KK_{S^1}(C(M ),\bb{C})$ are isomorphic by Thom isomorphism as follows.


\begin{lem}\label{cplx str of normal}
$\nu(M )$ has a complex structure. 
\end{lem}
\begin{proof}
We can define a skew-symmetric operator $d: \nu(M )\to \nu(M )$ by
$d:=\left.\frac{d}{d\theta}\right|_{\theta=0}\rho(e^{\sqrt{-1}\theta})$. 
Since $\nu(M )_m$ has no component of trivial representation, $J:=d/|d|$ is a complex structure.
\end{proof}

In particular, $\nu(M)$ is $K$-oriented.
Since $U_\delta$ is properly homotopy equivalent to the total space of $\nu(M )$, we have the following.

\begin{cor}
$C_0(U_\delta)$ and $C(M)$ are $KK_{S^1}$-equivalent.
\end{cor}

Since the Thom class for this $KK_{S^1}$-equivalence plays an important role in the present paper, we explicitly define it.
Since $U_\delta$ is diffeomorphic to an open set of $\nu(M )$ by the normal exponential mapping $\exp^\perp$, we can define a projection $\varpi:U_\delta\to M $.
We define the {\bf fiberwise Clifford operator} $C^{\rm fib}_x$ as follows.
For $x\in M $ and $v\in U_\delta$ so that $\varpi(v)=x$, $C^{\rm fib}_x(v)\in T_v^{\rm fib}U_\delta$ is defined by $-(d\exp^\perp_x)_{(\exp^\perp_x)^{-1}(v)}(\exp^\perp_x)^{-1}(v)$.
Roughly speaking, $C^{\rm fib}_x(v)=$``$\overrightarrow{xv}$''.
The field $\{C^{\rm fib}_x\}_{x\in X}$ defines a single operator $C^{\rm fib}$.

\begin{dfn}
$(1)$ When we regard $\nu(M )$ as a complex vector bundle, we denote it by $\nu_{\bb{C}}(M )$.

$(2)$ We define a $KK$-element
$[\tau_{U_\delta}]\in KK_{S^1}(C(M ),C_0(U_\delta))$ by
$$[\tau_{U_\delta}]=
\bra{C_0(U_\delta,\bigwedge^*\varpi^*\nu_{\bb{C}}(M )),\pi,\frac{C^{\rm fib}}{\sqrt{\delta^2-(C^{\rm fib})^2}}}.$$
\end{dfn}
\begin{rmk}\label{delta is not important}
We can flexibly choose $\delta$. If $0<\vep<\delta$, 
$\ca{R}KK(M ;C(M ),C_0(U_\vep))$ is isomorphic to $\ca{R}KK(M ;C(M ),C_0(U_\delta))$, and $[\tau_{U_\vep}]$ corresponds to $[\tau_{U_\delta}]$.
\end{rmk}


We want to give the inverse of $[j^*]$.
Note that $[j^*]$ and $[\tau_{U_\delta}]$ give opposite direction homomorphisms. Although they are not mutually inverse, $[\tau_{U_\delta}]$ gives an isomorphism.
Thus, in order to give ``$[j^*]^{-1}$'', we need a correction term.
It should be a $KK$-element $[\alpha]\in \ca{R}KK_{S^1}(M ;C(M ),C(M ))$ satisfying that 
$$[\alpha]\grotimes [\tau_{U_\delta}]\grotimes [j_*]=1.$$
Since $[\tau_{U_\delta}]\grotimes [j_*]$ is the restriction of the Thom class to the zero section, 
it is the Euler class, which is denoted by $[e_{U_\delta}]$.
Therefore, ``$[j^*]^{-1}$'' should be given by the composition of the Thom isomorphism and the ``inverse of the Euler class''. 

In order to define it, we introduce the following.
Recall that the representation ring $R(S^1)$ is isomorphic to $\bb{Z}[q,q^{-1}]$.
Each $q^k$ corresponds to the irreducible representation of weight $k$.

\begin{dfn}
$(1)$ Let $R(S^1)_\pos$ be the $R(S^1)$-algebra defined by
$$\bbra{\sum_n a_nq^n\ \middle|\  a_n\in \bb{Z}, a_n=0\text{ for all }n\ll 0}.$$

$(2)$ For an $R(S^1)$-module $\ca{M}$, we denote $\ca{M}\otimes_{R(S^1)}R(S^1)_\pos$ by $\ca{M}_\pos$. 
For $R(S^1)$-modules $\ca{M}$ and $\ca{N}$ and an $R(S^1)$-module homomorphism $f:\ca{M}\to \ca{N}$, the corresponding homomorphism $\ca{M}_\pos\to \ca{N}_\pos$ is denoted by $f_\pos$.

$(3)$ For an $R(S^1)$-module $\ca{M}$, we define a natural map $\pos: \ca{M}\to \ca{M}_\pos$ by $\pos(m):=m\otimes 1$.
\end{dfn}

The above algebra comes from the concept of positive energy representation of loop groups \cite{PS}. 

There are many invertible elements in this algebra.

\begin{lem}\label{fps invertible}
Let $Z=\sum_n a_nq^n\in R(S^1)_\pos$ and let $N$ be the least number such that $a_N\neq 0$. 
Then, $Z$ is invertible if and only if $a_N$ is $1$ or $-1$.
\end{lem}
\begin{proof}
If $a_N=\pm1$, $Z$ can be written as $\pm q^N+\sum_{n>N}a_nq^n=\pm q^N\bra{1\pm \sum_{n>N}a_nq^{n-N}}$.
Since $\pm q^N$ is invertible, it suffices to check the statement for $N=0$ and $a_N=1$.
Now it it clear by the Neumann series argument: $\bra{1- \sum_{n>N}a_nq^{n-N}}^{-1}=\sum_{k\geq 0}\bra{\sum_{n>N}a_nq^{n-N}}^k$.
\end{proof}




We prepare a basic fact from equivariant $K$-theory.

\begin{lem}\label{KH for trivial action}
If a compact group $H$ acts on a locally compact Hausdorff space $Y$ trivially, $K^*_H(Y)\cong K^*(Y)\otimes R(H)$
by the following correspondence: For an $H$-equivariant vector bundle $E$,
$$E\mapsto \sum_{\rho\in\widehat{H}} \Hom_H(E,V_\rho)\otimes \rho
\in K^*(Y)\otimes R(H),$$
where $V_\rho$ is the representation space corresponding to an irreducible representation $\rho$, and $\widehat{H}$ is the set of isomorphism classes of irreducible unitary representations.
\end{lem}

\begin{rmk}
The same construction works for more general $KK$-groups, for example $KK_H(A,B)\cong KK(A,B)\otimes R(H)$ for $C^*$-algebras $A$ and $B$ on which $H$ trivially acts.
\end{rmk}

Since the $S^1$-action on $M $ is trivial, we can apply this result on the equivariant bundle $\nu_{\bb{C}}(M )$.
There exists a $\bb{Z}_2$-graded vector bundle $E_n$ for each $n>0$ so that
\begin{equation}\label{decomposition of Euler 1}
\nu_{\bb{C}}(M )\cong \bigoplus_{n>0}E_n\otimes  \bb{C}_n,
\end{equation}
where $\bb{C}_n$ is the representation space of $S^1$ of weight $n$.
It is an essential point that $E_n=0$ for $n\leq 0$.
It holds because we have defined the complex structure using the $S^1$-action. 

\begin{lem}\label{lem comp of Eul}
$(1)$ As a $KK$-element, $[e_{U_\delta}]=\bra{C(M ,\bigwedge^*\bigoplus_{n>0}E_n\otimes \bb{C}_n),\pi,0}$.

$(2)$ When we decompose $\bigwedge^*\bigoplus_{n>0}E_n\otimes \bb{C}_n$ by Lemma \ref{KH for trivial action}, it is of the form
$$(\ud{\bb{C}}_{M }\otimes \bb{C}_0)\oplus
\bigoplus_{n>0}(\text{\rm a vector bundle})\otimes \bb{C}_n.$$

$(3)$ $[e_{U_\delta}]$ is invertible.
\end{lem}
\begin{proof}
$(1)$ Obvious.

$(2)$ We compute the exterior product.
\begin{align*}
\bigwedge^*\bigoplus_{n>0}E_n\otimes \bb{C}_n
&\cong \bigotimes_{n>0}\bigwedge^*E_n\otimes \bb{C}_n \\
&=\bigotimes_{n>0}\bigoplus_{0\leq m\leq \rank (E_n)}
\bra{\bigwedge^mE_n}\otimes \bb{C}_{nm} \\
&=\bigotimes_{n>0}\bbra{
(\ud{\bb{C}}_{M }\otimes \bb{C}_0)\oplus \text{ higher terms}} \\
&=(\ud{\bb{C}}_{M }\otimes \bb{C}_0)\oplus \text{ higher terms},
\end{align*}
where ``higher terms'' means a finite sum of ``a vector bundle $\otimes$ $\bb{C}_k$ for $k> 0$''s.
In the last equality, we use the fact that the product of a higher term and the trivial bundle is again higher,
and the product of higher terms is again higher.

$(3)$ It is parallel to Lemma \ref{fps invertible}.
\end{proof}

\begin{dfn}
We call the inverse of $[e_{U_\delta}]$ the {\bf inverse Euler class}.
We denote it by $[e^{-1}_{U_\delta}]\in KK_{S^1}(C(M ),C(M ))_\pos$.
\end{dfn}

The construction so far is valid for arbitrary $\delta>0$ if it is less than the injectivity radius.
From now on, we denote $\delta$ by $2\vep$.
$U_\vep$ will appear in the next subsection. There we will use both $U_\vep$ and $U_{2\vep}$.

Let us define an analytic index homomorphism for $X$.

\begin{dfn}
We define the {\bf localized index} by the composition of the following homomorphisms
$$KK_{S^1}(C_0(X),\bb{C})
\xrightarrow{[k_*]\grotimes -}
KK_{S^1}(C_0(U_{2\vep}),\bb{C})
\xrightarrow{[\tau_{U_{2\vep}}]\grotimes -}
KK_{S^1}(C(M ),\bb{C})
\xrightarrow{[e^{-1}_{U_{2\vep}}]\grotimes-}$$
$$KK_{S^1}(C(M ),\bb{C})_\pos
\xrightarrow{[\ud{\bb{C}}_{M }]\grotimes-}
KK_{S^1}(\bb{C},\bb{C})_\pos=R(S^1)_\pos.$$

The composition of this index homomorphism is denoted by $\ind^\pos_{S^1}$.
\end{dfn}


\begin{pro}
The localized index is independent of $\vep$.
\end{pro}
\begin{proof}
Let $\delta$ be another positive real number so that $2\delta$ is less than the injectivity radius of $\exp^\perp$.
Then, clearly $[e^{-1}_{U_{2\vep}}]=[e^{-1}_{U_{2\delta}}]$. 
Thus, it suffices to prove that the following diagram commutes:
$$\begin{CD}
KK_{S^1}(C_0(X),\bb{C}) 
@>{[k_*]\grotimes -}>> 
KK_{S^1}(C_0(U_{2\vep}),\bb{C}) \\
@V[k'_*]\grotimes -VV 
@VV{[\tau_{U_{2\vep}}]\grotimes -}V \\
KK_{S^1}(C_0(U_{2\delta}),\bb{C}) 
@>>{[\tau_{U_{2\delta}}]\grotimes -}> KK_{S^1}(C(M ),\bb{C}),
\end{CD}$$
where $k'_*:C_0(U_{2\delta})\to C_0(X)$ is the zero extension.

We may assume that $\delta>\vep$ by symmetry.
For a representative of the Thom class of $U_{2\delta}$, we can use
$$\bra{C_0(U_{2\vep},\bigwedge^*\varpi^*\nu_{\bb{C}}(M^{S^1})),\pi,\frac{C^{\rm fib}}{\sqrt{4\vep^2-(C^{\rm fib})^2}}}$$
instead of the natural one $\bra{C_0(U_{2\delta},\bigwedge^*\varpi^*\nu_{\bb{C}}(M^{S^1})),\pi,\frac{C^{\rm fib}}{\sqrt{4\delta^2-(C^{\rm fib})^2}}}$.
Then, the same Kasparov module satisfies the condition to be a Kasparov product of 
$[\tau_{U_{2\vep}}]$ and $[k_*]$ and that of $[\tau_{U_{2\delta}}]$ and $[k'_*]$.
\end{proof}

This localized index satisfies the following property.

\begin{thm}
For $[D]\in KK_{S^1}(C(M),\bb{C})$, we have
$[\ud{\bb{C}}_{M}]\grotimes [D]
=\ind_{S^1}^\pos([i^*]\grotimes [D])$.
\end{thm}
\begin{proof}
Obvious from the construction.
\end{proof}


\subsection{Fixed-point formula of the localized index}\label{Fixed-point formula of the localized index}

We will deduce the fixed-point formula for the localized index by using Theorem \ref{thm inv ind thm}.
We first study the general cases, and then the $K$-oriented cases.

\subsubsection{General cases}

We will prove the fixed-point formula by translating each step of the construction of the localized index, into the topological language, by using Definition \ref{Def PD}.
The problem is only the following: Although we assumed that the base manifold is complete in Definition \ref{Def PD}, $U_{2\vep}$ is not complete.
In order to overcome this problem, we introduce the following.
Recall that $U_x$ is the $\vep$-neighborhood of $x$ in $M$.
Consequently, we have $U_x\subseteq U_{2\vep}$ if $x\in U_\vep$.

Recall that for a Riemannian manifold $Y$, we define $Cl_\tau(Y):=C_0(Y,\Cl_+(TY))$.


\begin{dfn}\label{PD for non complete}
$(1)$ We define the local Bott element of $U_{\vep}$ by
$$[\Theta_{U_\vep}]':=
\bbra{
\Bigl(C_0(U_x)\grotimes \Cl_+(T_xX),1_x,\vep^{-1}\Theta_x
\Bigr)}_{x\in U_{\vep}}\in \ca{R}KK_{S^1}(U_\vep;C_0(U_\vep),C_0(U_{2\vep})\grotimes \uwave{Cl_\tau(U_{\vep})}).$$

$(2)$ We define $\PD_{U_\vep}':KK_{S^1}(C_0(U_{2\vep}),\bb{C})
\to \ca{R}KK_{S^1}(U_\vep;C_0(U_{\vep}),Cl_\tau(U_\vep))$ by
$$\PD_{U_\vep}'([D]):=
[\Theta_{U_\vep}]'\grotimes_{C_0(U_{2\vep})}{[D]}.$$
%
%
%
\end{dfn}
\begin{rmks}
$(1)$ Note that $[\Theta_{U_\vep}]'$ cannot be defined without data of $U_{2\vep}$.

$(2)$ $\PD_{U_\vep}'$ is a homomorphism from equivariant $K$-homology of $U_{2\vep}$ to equivariant representable $K$-theory {\it of the smaller space $U_\vep$}, not of $U_{2\vep}$.
Although this seems to be strange, since $U_{2\vep}$ and $U_{\vep}$ are equivariantly homotopy equivalent, we have
$$\ca{R}KK_{S^1}(U_{2\vep};C_0(U_{2\vep}),Cl_\tau(U_{2\vep}))\cong \ca{R}KK_{S^1}(U_\vep;C_0(U_{\vep}),Cl_\tau(U_\vep)).$$
\end{rmks}

We need one more ingredient in order to deduce the cohomology formula.

\begin{dfn}\label{t ind for tri act}
Let $Y$ be a compact manifold equipped with the trivial action of a compact Lie group $G$.

$(1)$ We define $ch:\ca{R}KK_G(Y;C(Y),C_0(TY))\to H^*_c(TY;\bb{Q})\otimes R(G)$ by the composition
$$\ca{R}KK_G(Y;C(Y),C_0(TY))
\xrightarrow{\cong}
K_{c}^0(TY)\otimes R(G)
\xrightarrow{ch\otimes \id}
H^*_c(TY;\bb{Q})\otimes R(G).$$

$(2)$ We define $\displaystyle \int_{TY}:H^*_c(TY;\bb{Q})\otimes R(G)\to \bb{Q}\otimes R(G)$ by $\displaystyle \int_{TY} (u\otimes \rho):=\bra{\int_{TY}u}\otimes \rho$ for $u\in H^*_c(TY;\bb{Q})$ and $\rho \in R(G)$, where $TY$ is oriented by the natural identification $T(TY)\cong TY\otimes \bb{C}$.

$(3)$ We define 
$t-\ind_Y:\ca{R}KK_G(Y;C(Y),C_0(TY))\to \bb{Q}$ by
$$t-\ind_Y(u):=(-1)^{\dim(Y)}\int_{TY}ch(u)\td(TY\otimes\bb{C}),$$
where $\td(TY \otimes \bb{C})$ is the Todd class of $TY \otimes \bb{C}$.
\end{dfn}

\begin{thm}[Atiyah-Singer]\label{thm AS}
If $Y$ is a compact manifold equipped with the trivial action of a compact Lie group $G$, for a Dirac operator on $D$ on $Y$, we have
$$[\ud{\bb{C}}_Y]\grotimes[D]
=t-\ind_Y(\sigma(D)).$$
\end{thm}

Recall that $[\ca{B}_{TX}]\in \ca{R}KK_{S^1}(X;Cl_\tau(X),C_0(TX))$ gives an $\ca{R}KK$-equivalence between $Cl_\tau(X)$ and $C_0(TX)$. 
We have defined everything appearing in the following diagram.
{\footnotesize
\begin{equation}\label{Big diagram cpt}
\xymatrix{
KK_{S^1}(C_0(X),\bb{C})
\ar^-{[\Theta_X]\grotimes-}[r]
\ar_-{[k_*]\grotimes-}[d] 
\ar@{}|-{(i)}[rdd]&
\ca{R}KK_{S^1}(X;C_0(X),Cl_\tau(X))
\ar^-{-\grotimes[\ca{B}_{TX}]}[r]
\ar|-{k^*}[dd] 
\ar@{}|-{(ii)}[rdd] &
\ca{R}KK_{S^1}(X;C_0(X),C_0(TX))
\ar^-{(Tk)^*}[dd]\\
KK_{S^1}(C_0(U_{2\vep}),\bb{C})
\ar^-{[\Theta_{U_\vep}]'\grotimes-}[rd]
\ar_-{[\tau_{U_{2\vep}}]\grotimes-}[ddd]
\ar@{}|-{(iii)}[rddd] & & & \\
&
\ca{R}KK_{S^1}(U_{\vep};C_0(U_{\vep}),Cl_\tau(U_{\vep}))
\ar^-{-\grotimes[\ca{B}_{TU_{\vep}}]}[r]
\ar|-{j^*}[d] 
\ar@{}|-{(iv)}[rd] &
\ca{R}KK_{S^1}(U_{\vep};C_0(U_{\vep}),C_0(TU_{\vep}))
\ar^-{(Tj)^*}[d] \\
&
\ca{R}KK_{S^1}(M ;C(M ),Cl_{TU_{\vep}}(M ))
\ar^-{-\grotimes[\ca{B}_{TU_\vep}|_M]}[r]
\ar|-{-\grotimes[S_{\nu(M )}^*]}[d] 
\ar@{}|-{(v)}[rd] &
\ca{R}KK_{S^1}(M ;C(M ),C_0(TU_{\vep}|_{M })) 
\ar^-{-\grotimes[d^{Spin^c}_{\nu(M )}]}[d] \\
%
KK_{S^1}(C(M ),\bb{C})
\ar^-{[\Theta_{M }]\grotimes-}[r]
\ar_-{[e^{-1}_{U_{2\vep}}]\grotimes-}[d] 
\ar@{}|-{(vi)}[rd] &
\ca{R}KK_{S^1}(M ;C(M ),Cl_\tau(M ))
\ar^-{-\grotimes[\ca{B}_{TM }]}[r]
\ar|-{[e^{-1}_{U_{\vep}}]\grotimes-}[d] 
\ar@{}|-{(vii)}[rd]&
\ca{R}KK_{S^1}(M ;C(M ),C_0(TM )) 
\ar^-{[e^{-1}_{U_{\vep}}]\grotimes-}[d] \\
KK_{S^1}(C(M ),\bb{C})_\pos
\ar_-{[\Theta_{M }]\grotimes-}[r]
\ar_-{[\ud{\bb{C}}_{M }]\grotimes-}[d] 
\ar@{}|-{(viii)}[rrd]&
\ca{R}KK_{S^1}(M ;C(M ),Cl_\tau(M ))_\pos
\ar_-{-\grotimes[\ca{B}_{TM }]}[r] &
\ca{R}KK_{S^1}(M ;C(M ),C_0(TM ))_\pos
\ar^-{t-\ind_{M }}[d]  \\
R(S^1)_\pos
\ar[rr] &&
R(S^1)_\pos\otimes\bb{Q}.
}
\end{equation}}

Note that $[e^{-1}_{U_\vep}]=[e^{-1}_{U_{2\vep}}]$, and we do not distinguish them from now on.

\begin{pro}
Diagram (\ref{Big diagram cpt}) commutes.
\end{pro}
\begin{proof}
We will prove that each square commutes.

$(i)$ Let $[D]\in KK_{S^1}(C_0(X),\bb{C})$. 
We note that 
$$k^*([\Theta_X]\grotimes [D])=\bbra{(C_0(U_x)\otimes \Cl_+(T_xX),1_x,\vep^{-1}\Theta_x)\grotimes_{C_0(X)} [D]}_{x\in U_\vep}.$$
Since the right action of $C_0(X)$ on $C_0(U_x)$ is the extension of the right action of $C_0(U_{2\vep})$, we have
\begin{align*}
&\{(C_0(U_x)\otimes \Cl_+(T_xX),1_x,\vep^{-1}\Theta_x)\grotimes_{C_0(X)} [D]\}_{x\in X} \\
&\ \ \ =\{(C_0(U_x)\otimes \Cl_+(T_xX),1_x,\vep^{-1}\Theta_x)\grotimes_{C_0(U_{2\vep})} [k_*]\grotimes_{C_0(X)} [D]\}_{x\in X}.
\end{align*}
It is nothing but a representative of $[\Theta_{U_\vep}]'\grotimes[k_*]\grotimes [D]$ by associativity of Kasparov product.
Similarly, commutativity of $(ii)$ and $(iv)$ can be easily checked and we leave it to the reader.


$(vi)$ Let $[D]\in KK_{S^1}(C(M ),\bb{C})$. 
Since 
$$[\Theta_{M }]\grotimes [e^{-1}_{U_\vep}]\grotimes[D]
=\bbra{(C_0(U_x)\otimes \Cl_+(T_xX),1_x,\vep^{-1}\Theta_x)\grotimes [e^{-1}_{U_\vep}]
\grotimes_{C_0(X)} [D]}_{x\in U_\vep},$$
we may consider the following instead of it:
$$\bbra{(C_0(U_x)\grotimes \Cl_+(T_xX),1_x,\vep^{-1}\Theta_x)\grotimes_{C_0(U_x)} [e^{-1}_{U_\vep}|_{U_x}]
\grotimes_{C_0(U_x)} [D|_{U_x}]}_{x\in U_\vep}.$$
Then, by the exponential mapping, we can canonically trivialize $[e^{-1}_{U_\vep}|_{U_x}]$, and we can construct an isomorphism between the above and
$$\bbra{(C_0(U_x)\grotimes e^{-1}_{U_\vep}|_{x}\grotimes \Cl_+(T_xX),1_x,\vep^{-1}\Theta_x)
\grotimes_{C_0(X)} [D|_{U_x}]}_{x\in U_\vep}
=[e^{-1}_{U_\vep}]\grotimes[\Theta_{M }]\grotimes [D].$$
Similarly, commutativity of $(vii)$ can be easily checked.


$(v)$ Since
$[\ca{B}_{TU_\vep}|_{M}]
=[\ca{B}_{TM}]\grotimes_{C(M )}[\ca{B}_{\nu(M )}]$, we have
\begin{align*}
[\ca{B}_{TU_\vep}|_{M}]\grotimes [d^{Spin^c}_{\nu(M )}]
&=[\ca{B}_{TM}]\grotimes[\ca{B}_{\nu(M )}]\grotimes [d_{\nu(M )}]\grotimes [S^*_{\nu(M )}] \\
&=[\ca{B}_{TM}]\grotimes[S^*_{\nu(M )}],
\end{align*}
the square $(v)$ commutes.

Commutativity of $(viii)$ is deduced from Theorem \ref{thm AS} and Theorem \ref{thm inv ind thm}.

$(iii)$ It is the most complicated.
We compute $j^*([\Theta_{U_\vep}]')\grotimes[S_{\nu(M )}^*]$ and $[\Theta_{M }]\grotimes [\tau_{U_{2\vep}}]$. 
The essential point is that the following are homeomorphic:
a ``$\dim(X)$-ball'' and ``the direct product of a $\dim(M )$-ball and a $\rank(\nu(M ))$-ball'', where the notion ``$n$-ball'' stands for an $n$-dimensional ball.

In order to distinguish $\Theta_x$'s appearing in the two local Bott elements, we denote that for $U_{2\vep}$ by $\Theta^{X}_x$ and that for $M $ by $\Theta_x^{M}$.
In order to distinguish $\vep$-neighborhoods appearing in the two local Bott elements, we denote that for $U_{2\vep}$ by $U_x$ and that for $M $ by $V_x$. 
Since $U_{2\vep}$ has a structure of a fiber bundle, 
$U_{2\vep}|_{V_x}=\varpi^{-1}(V_x)$ makes sense.
It is diffeomorphic to the direct product of the $\vep$-neighborhood of $x$ in $M $ and a ball of radius $2\vep$ centered at the origin of $\nu(M )|_x$.
The situation is summarized as follows:
$$\xymatrix{
X 
&U_{2\vep}
 \ar@{_{(}->}[l] & 
M  
 \ar@{_{(}->}[l] \\
&U_{2\vep}|_{V_x}
 \ar^{\varpi}@{->>}[rd]
 \ar@{_{(}->}[u] \\
&U_{x}
 \ar@{_{(}->}[u] &
V_x.
 \ar@{_{(}->}[l]
 \ar@{_{(}->}[uu] }.$$

As a preliminary, we prepare a more convenient representative of $[\tau_{U_{2\vep}}]$
$$\bra{C_0(U_{\vep},\bigwedge^*\varpi^*\nu_{\bb{C}}(M )),\pi,\frac{C^{\rm fib}}{\sqrt{\vep^2-(C^{\rm fib})^2}}},$$
instead of the natural one.

Note that $[\Theta_X]$ is represented by, in the bounded picture,
$$\bbra{\bra{C_0(U_x)\grotimes \Cl_+(T_xX),1_x,\frac{\Theta^X_x}{\sqrt{\vep^2-(\Theta^X_x)^2}}}}_{x\in X}.$$

First, we compute $[\Theta_{M }]\grotimes [\tau_{U_{2\vep}}]$ using the above representatives.
In the family description for $\ca{R}KK$-theory, the Kasparov product $[\Theta_{M }]\grotimes [\tau_{U_{2\vep}}]$ is given by
$$\bbra{\bra{
C_0(U_{\vep}|_{V_x},\bigwedge^*\nu_{\bb{C}}(M ))\grotimes \Cl_+(T_xM ),
1_x,{\frac{C^{\rm fib}_x}{\sqrt{\vep^2-(C^{\rm fib}_x)^2}}+\frac{\Theta_x^{M}}{\sqrt{\vep^2-(\Theta_x^{M})^2}}}}}_{x\in M }.$$

%

Second, we modify the above operator. 
We consider two ``bounded transformations''.
Let $\fra{b}(x):=\frac{x}{\sqrt{1+x^2}}$.
Let $\rho$ be a function on $\bb{R}$ given by
$$\rho(x)=\begin{cases}
\frac{x}{\vep} &(|x|\leq \vep), \\
\frac{x}{|x|} &(|x|>\vep).
\end{cases}$$
Let $F:=\fra{b}\bra{\frac{C^{\rm fib}_x}{\sqrt{\vep^2-(C^{\rm fib}_x)^2}}+\frac{\Theta_x^{M}}{\sqrt{\vep^2-(\Theta_x^{M})^2}}}$.
We prove that 
$\bbra{
C_0(U_{\vep}|_{V_x},\bigwedge^*\nu_{\bb{C}}(M ))\grotimes \Cl_+(T_xM ),
1_x,F}_{x\in M }$
and $\bbra{
C_0(U_{\vep}|_{V_x},\bigwedge^*\nu_{\bb{C}}(M ))\grotimes \Cl_+(T_xM ),
1_x,\rho(\Theta^X_x)}_{x\in M }$ are operator homotopic.
We construct a homotopy $F_s$ of Clifford algebra-valued function on $U_{\vep}|_{V_x}$. 
Let $f_s:=(1-s)F+s\vep^{-1}\Theta^X_x$ and $F_s:=\rho(f_s)$.
We prove that $F_s$ gives an operator homotopy.
Since $F_s$ gives a left multiplication with a Clifford algebra-valued function, it commutes with the right multiplication. 
It is clear that $F_s$ is $S^1$-equivariant.
Thus, what we essentially need to prove is that $F_s^2-1$ vanishes on the boundary of $U_{\vep}|_{V_x}$.

Let us study $f_s$. Since $\Theta^X_x=C^{\rm fib}_x+\Theta_x^{M}$,
\begin{align*}
f_s&=\bra{\bra{1-s}
\bra{1+\frac{(C^{\rm fib}_x)^2}{\vep^2-(C^{\rm fib}_x)^2}+\frac{(\Theta_x^{M})^2}{\vep^2-(\Theta_x^{M})^2}}^{-1/2}\frac{1}{\sqrt{\vep^2-(C^{\rm fib}_x)^2}}+s\vep^{-1}}C^{\rm fib}_x \\
&\ \ \ \ \ \ +
\bra{\bra{1-s}
\bra{1+\frac{(C^{\rm fib}_x)^2}{\vep^2-(C^{\rm fib}_x)^2}+\frac{(\Theta_x^{M})^2}{\vep^2-(\Theta_x^{M})^2}}^{-1/2}\frac{1}{\sqrt{\vep^2-(\Theta_x^{M})^2}}+s\vep^{-1}}\Theta_x^{M}.
\end{align*}
We denote it by $f_s=A_sC^{\rm fib}_x+B_s\Theta_x^{M}$.
We need to check that $F_s^2-1=\rho(f_s)^2-1=\rho(f_s^2-1)$ vanishes on the boundary.
Note that $f_s^2$ is scalar-valued.
Thus, thanks to the definition of $\rho$, it suffices to prove that $f_s^2\geq 1$ on the boundary.
Since $f_s^2=A_s^2(C^{\rm fib}_x)^2+B_s^2(\Theta_x^{M})^2$ and $F_0=F$, it suffices to check that $A_s\geq A_0$ and $B_s\geq B_0$, that is to say,
\begin{align*}
&\bra{1-s}
\bra{1+\frac{(C^{\rm fib}_x)^2}{\vep^2-(C^{\rm fib}_x)^2}+\frac{(\Theta_x^{M})^2}{\vep^2-(\Theta_x^{M})^2}}^{-1/2}\frac{1}{\sqrt{\vep^2-(C^{\rm fib}_x)^2}}+s\vep^{-1} \\
&\ \ \ \ \ \ \geq 
\bra{1+\frac{(C^{\rm fib}_x)^2}{\vep^2-(C^{\rm fib}_x)^2}+\frac{(\Theta_x^{M})^2}{\vep^2-(\Theta_x^{M})^2}}^{-1/2}\frac{1}{\sqrt{\vep^2-(C^{\rm fib}_x)^2}}
\end{align*}
\begin{align*}
&\bra{1-s}
\bra{1+\frac{(C^{\rm fib}_x)^2}{\vep^2-(C^{\rm fib}_x)^2}+\frac{(\Theta_x^{M})^2}{\vep^2-(\Theta_x^{M})^2}}^{-1/2}\frac{1}{\sqrt{\vep^2-(\Theta_x^{M})^2}}+s\vep^{-1} \\
&\ \ \ \ \ \ \geq 
\bra{1+\frac{(C^{\rm fib}_x)^2}{\vep^2-(C^{\rm fib}_x)^2}+\frac{(\Theta_x^{M})^2}{\vep^2-(\Theta_x^{M})^2}}^{-1/2}\frac{1}{\sqrt{\vep^2-(\Theta_x^{M})^2}}
\end{align*}
on the boundary. 
We discuss only the former one.
The above inequality is equivalent to
$$\vep^{-1}\geq 
\bra{1+\frac{(C^{\rm fib}_x)^2}{\vep^2-(C^{\rm fib}_x)^2}+\frac{(\Theta_x^{M})^2}{\vep^2-(\Theta_x^{M})^2}}^{-1/2}\frac{1}{\sqrt{\vep^2-(C^{\rm fib}_x)^2}}$$
on the boundary.
It is clear from a direct calculation.


Finally, we prove that the diagram commutes.
$j^*([\Theta_{U_\vep}]')\grotimes[S_{\nu(M )}^*]$
is given by
$$\bbra{\bra{C_0(U_x,\wedge^*\varpi^*\nu_{\bb{C}}(M ))\grotimes \Cl_+(T_xM ),1_x,\vep^{-1}\Theta^X_x}}_{x\in M }.$$
Let us construct a homotopy between these Kasparov modules.
Let 
$$E_x:=\bbra{e:[0,1]\to C_0(U_{\vep}|_{V_x},\bigwedge^*\varpi^*\nu_{\bb{C}}(M ))\grotimes \Cl_+(T_xM )\ \midd \ e(1)\text{ vanishes on }U_{\vep}|_{V_x}\setminus U_x}.$$
Then, an $\ca{R}KK_{S^1}$-cycle,
$$\bbra{E_x,1_x,\rho(\Theta_x^X)}_{x\in M }$$
gives a homotopy between the two $\ca{R}KK_{S^1}$-cycles. This is because the restriction of $\rho(\Theta_x^X)$ to $C_0(U_x,\wedge^*\nu_{\bb{C}}(M ))\grotimes \Cl_+(T_xM )$ is $\vep^{-1}\Theta^X_x$.
\end{proof}

This proposition tells us that the localized index can be computed by the composition of the right vertical arrows.
We call it the {\bf topological localized index}.

\begin{dfn}
$t-\ind^\pos_{S^1}:\ca{R}KK_{S^1}(X;C_0(X),C_0(TX))\to \bb{Q}\otimes R(S^1)_\pos$ is defined by the composition
$$\ca{R}KK_{S^1}(X;C_0(X),C_0(TX))
 \xrightarrow{(Tk)^*}
\ca{R}KK_{S^1}(U_{\vep};C_0(U_{\vep}),C_0(TU_{\vep}))
 \xrightarrow{(Tj)^*} $$
$$\ca{R}KK_{S^1}(M ;C(M ),C_0(TU_{\vep}|_{M })) 
 \xrightarrow{-\grotimes[d^{Spin^c}_{\nu(M )}]}
\ca{R}KK_{S^1}(M ;C(M ),C_0(TM ))$$
$$
\xrightarrow{[e^{-1}_{U_{\vep}}]\grotimes-}
\ca{R}KK_{S^1}(M ;C(M ),C_0(TM ))_\pos
 \xrightarrow{t-\ind_{M }}
\bb{Q}\otimes R(S^1)_\pos.
$$
\end{dfn}

Now the following is clear.

\begin{thm}\label{fpt for non cpt}
For $[D]\in KK_{S^1}(C_0(X),\bb{C})$, we have $\ind_{S^1}^\pos([D])=
t-\ind_{S^1}^\pos(\sigma(D)).$
\end{thm}

\subsubsection{$K$-oriented cases}\label{K-ori noncpt}

When $X$ is $K$-oriented, we can give a more convenient formula to compute $\ind_{S^1}^\pos$ by using the isomorphism
$$\ca{R}KK_{S^1}(X;C_0(X),C_0(TX))
\cong \ca{R}KK_{S^1}(X;C_0(X),C_0(X))
\cong K_{S^1}^0(X).$$

Let $S_X$ be an $S^1$-equivariant Spinor bundle of $X$. 
Since the normal bundle admits a complex structure, it has a natural $S^1$-equivariant Spinor bundle $\wedge^*{\nu_{\bb{C}}(M)}$.
Thus, $M $ is also of $K$-orientable.
Let $S_{M }$ be a Spinor bundle of $M $ and $L$ be a $\bb{Z}_2$-graded\footnote{A $\bb{Z}_2$-graded line bundle is a pair of a line bundle and a locally constant function taking values in $\bb{Z}_2$. This concept is explained in  \cite[Definition 2.1]{FHTI}.} complex line bundle $L$ so that $S_X|_{M }\cong S_{M }\grotimes \wedge^*\nu_{\bb{C}}(M )\grotimes L$.\footnote{
$S_X$ automatically gives a Spinor bundle of $M $ by $S_{M}'=\Hom_{\Cl(\nu(M ))}(\wedge^*\nu_{\bb{C}}(M ),S_X)$.
This Spinor satisfies a convenient relation $S_X\cong S_{M}'\otimes \wedge^*\nu_{\bb{C}}(M )$.
However, it is more convenient for computation to use favorite Spinor bundle over $M $.}
We call this line bundle the {\bf difference line bundle for $X$}.
The corresponding $\ca{R}KK$-element is denoted by
$$[L]\in \ca{R}KK_{S^1}(M ;C(M ),C(M )).$$
The dual bundle $L^*$ gives the inverse of $L$, namely $L\grotimes L^*\cong \ud{\bb{C}}_M$.

Let us rewrite Definition \ref{t ind for tri act} for a $K$-oriented manifold.

\begin{dfn}
Let $Y$ be a $K$-oriented manifold equipped with the trivial action of a compact Lie group $G$. 


$(1)$ We define $ch:\ca{R}KK_G(Y;C(Y),C(Y))\to H^*(Y;\bb{Q})\otimes R(G)$ by the composition
$$\ca{R}KK_G(Y;C(Y),C(Y))
\xrightarrow{\cong}
K(Y)\otimes R(G)
\xrightarrow{ch\otimes \id}
H^*(Y;\bb{Q})\otimes R(G).$$

$(2)$ We define $\displaystyle \int_{Y}:H^*(Y;\bb{Q})\otimes R(G)\to \bb{Q}\otimes R(G)$ by $\displaystyle \int_{Y} (u\otimes  \rho):=\bra{\int_{Y}u}\otimes \rho$
for $u\in H^*(Y;\bb{Q})$ and $\rho\in R(G)$.

$(3)$ We define 
$t-\ind_Y':\ca{R}KK_G(Y;C(Y),C(Y))\to \bb{Q}$ by
$$t-\ind_Y'(u):=\int_{Y}ch(u)\td(TY),$$
where $\td(TY)\in H^*(Y;\bb{Q})$ is the $Spin^c$-Todd class of $TY$. 
\end{dfn}

Then, $t-\ind_{M }'$ and $t-\ind_{M }$ are related in the following way.

\begin{lem}\label{K theory Thom}
For a compact $K$-oriented manifold $Y$, the following diagram commutes.
$$
\xymatrix{
K^0(TY) 
\ar^{-\grotimes[d^{Spin^c}_{TY}]}[rr]
\ar_{t-\ind_Y}[rd] &&
K^0(Y) 
\ar^{t-\ind_Y'}[ld] \\
& \bb{Q}. &}$$
\end{lem}

We want to give a cohomology formula of the localized using the integration on $X$, not on $TX$.
For this aim, we consider the following diagram.
\begin{equation}\label{Big diagram for K ori}
\xymatrix{
\ca{R}KK_{S^1}(X;C_0(X),C_0(TX))
\ar^-{-\grotimes[d_{TX}^{Spin^c}]}[r]
\ar_-{(Tk)^*}[d] 
\ar@{}|-{(i)}[rd] &
\ca{R}KK_{S^1}(X;C_0(X),C_0(X))
\ar^-{k^*}[d] \\
\ca{R}KK_{S^1}(U_{\vep};C_0(U_\vep),C_0(TU_{\vep}))
\ar_-{-\grotimes[d_{TU_{\vep}}^{Spin^c}]}[r]
\ar_-{j^*}[d] 
\ar@{}|-{(ii)}[rdd] &
\ca{R}KK_{S^1}(U_{\vep};C_0(U_\vep),C_0(U_{\vep})) 
\ar^{j^*}[d] \\
\ca{R}KK_{S^1}(M ;C(M ),C_0(TU_{\vep}|_{M }))
\ar_-{-\grotimes[d^{Spin^c}_{\nu(M )}]}[d] & 
\ca{R}KK_{S^1}(M ;C(M ),C(M ))
\ar^-{[L^*]\grotimes-}[d] \\
%
\ca{R}KK_{S^1}(M ;C(M ),C_0(TM ))
\ar^-{-\grotimes[d_{TM }^{Spin^c}]}[r]
\ar_-{[ e^{-1}_{U_\vep}]\grotimes-}[d] 
\ar@{}|-{(iii)}[rd] &
\ca{R}KK_{S^1}(M ;C(M ),C(M ))
\ar^-{[ e^{-1}_{U_\vep}]\grotimes-}[d] \\
\ca{R}KK_{S^1}(M ;C(M ),C_0(TM ))_\pos
\ar_-{-\grotimes[d_{T{M }}^{Spin^c}]}[r]
\ar_-{t-\ind_{M }}[d] 
\ar@{}|-{(iv)}[rd] &
\ca{R}KK_{S^1}(M ;C(M ),C(M ))_\pos
\ar^-{t-\ind_{M }'}[d] \\
R(S^1)_\pos\otimes\bb{Q}
\ar_{\id}[r] &
R(S^1)_\pos\otimes\bb{Q}.
}
\end{equation}

The squares $(i)$ and $(iii)$ clearly commute.
The square $(iv)$ commutes thanks to Lemma \ref{K theory Thom}.

Let us prove that $(ii)$ commutes.
We compute the composition of the following: $(-\grotimes[d_{TU_\vep}^{Spin^c}])^{-1}$, $j^*$, $-\grotimes[d^{Spin^c}_{\nu(M )}]$ and $-\grotimes[d_{TM }^{Spin^c}]$, where the composition is a map from 
$\ca{R}KK_{S^1}(U_\vep;C_0(U_\vep),C_0(U_\vep))$ to 
$\ca{R}KK_{S^1}(M;C(M),C(M))$.
By Theorem \ref{Thom isom E} ad Definition \ref{KK equiv K ori}, for $[E]\in \ca{R}KK_{S^1}(U_\vep;C_0(U_\vep),C_0(U_\vep))$, $(-\grotimes[d_{TU_\vep}^{Spin^c}])^{-1}([E])$ is given by $[E]\grotimes [\ca{B}_{TU_\vep}^{Spin^c}]$.
By $S_X|_M\cong S_M\grotimes \wedge^*\nu_\bb{C}(M)^*\grotimes L$, we have 
$[\ca{B}_{TU\vep}^{Spin^c}]=
[\ca{B}_{\nu_\bb{C}(M)}^{Spin^c}]\grotimes 
[\ca{B}_{TM}^{Spin^c}]\grotimes [L]$.
Therefore, $[d_{TX}^{Spin^c}]=
[d_{\nu_\bb{C}(M)}^{Spin^c}]\grotimes 
[d_{TM}^{Spin^c}]\grotimes [L^*]$.
Since $[\ca{B}_{TU_\vep}^{Spin^c}]\grotimes [d_{TU_\vep}^{Spin^c}]=1$, we have
$$j^*([E]\grotimes [\ca{B}_{TU_\vep}^{Spin^c}])\grotimes
[d^{Spin^c}_{\nu(M )}]\grotimes
[d_{TM }^{Spin^c}]
=j^*([E])\grotimes [L^*].$$

Consequently, we have the following fixed-point formula for the localized index.
The composition of the right vertical arrows of the above commutative diagram is denoted by $(t-\ind_{S^1}^{\pos})'$.


\begin{thm}\label{fixed-point formula non cpt}
$(1)$ For $[u]\in \ca{R}KK_{S^1}(X;C_0(X),C_0(TX))$, 
$(t-\ind_{S^1}^{\pos})'([u]\grotimes [d_{TX}^{Spin^c}])
=t-\ind_{S^1}^{\pos}([u])$.

$(2)$ In particular, for an $S^1$-equivariant $\bb{Z}_2$-graded Hermite bundle $F$ over $X$ and a Dirac operator $D_F$ on $S_X\grotimes F$,
$$
\ind^{\pos}_{S^1}([D_F])=
\int_{M }\td(TM )ch\bra{F|_{M }\grotimes L\grotimes {\sum_{l\geq 0}\bra{
-\sum_{k\geq 1}(-1)^k\wedge^k \nu_\bb{C}(M )}^l}}.$$
\end{thm}




\begin{ex}
Let $X=\bb{C}$ be a complex plane on which $S^1$ linearly acts with non-zero weight $k$, namely $e^{i\theta}\cdot z=e^{ik\theta}z$.
The fixed-point is only the origin: $M=\{O\}$.
We use the complex structure in order to define the Spinor bundle: $S_X=\wedge^* TX_\bb{C}$. 
The normal bundle is $\nu_\bb{C}(M )=T_OX=\bb{C}_{|k|}$.
Recall that we choose the complex structure on the normal bundle so that the weight is positive.
Thus, the difference line bundle for $X$ is given as follows
$$L=\begin{cases}
 q^0 & (k>0), \\
- q^{|k|} & (k<0). \end{cases}$$
Obviously $\td(TM )=1$. 

Let $[D]\in KK_{S^1}(C_0(X),\bb{C})$ be a $K$-homology element represented by a Dirac operator $D$ acting on a Clifford bundle $W$. Since $X$ is of $K$-oriented, it can be written as 
$F\grotimes S_X$.
Thanks to Lemma \ref{KH for trivial action} and the fact that $K^0(X)\cong \bb{Z}$, $F$ can be written as
$\sum_n c_n q^n$ for $c_n\in \bb{Z}$.


Let us compute the localized index of $[D]$.
If $k>0$,
\begin{align*}
\ind_{S^1}^{\pos}([D])
&=\sum_n c_n q^n{\sum_{l\geq 0}\bra{
-\sum_{m\geq 1}(-1)^m\wedge^m \nu_\bb{C}(M )}^l} \\
&=\sum_n c_n q^n\sum_{l\geq 0}{
 q^{kl}} \\
&=\sum_nc_n\bra{ q^{n}+ q^{n+k}+ q^{n+2k}+\cdots}.
\end{align*}

By the same argument, if $k<0$, the index is given by
$$-q^k\sum_nc_n\bra{ q^{n}+ q^{n+|k|}+ q^{n+2|k|}+\cdots}.$$

For example, if $F=\ud{\bb{C}}_{X}$ and $k=1$, the localized index is given by
$$ q^0+ q^1+ q^{2}+ q^{3}+\cdots.$$
\end{ex}

\subsection{A reformulation}\label{reform of fpt}

In the arguments so far, we have used several $C^*$-algebras which do not make sense for infinite-dimensional manifolds.
In order to construct an infinite-dimensional analogue of the localized index and the fixed-point formula, at least, we need to reformulate our index.
Concretely, we will replace a single $C^*$-algebra $C_0(X)$ for an $K$-oriented manifold $X$ with the graded suspension of $Cl_\tau(X)$, and we will replace $C_0(X)$-$C^*$-algebras (which we regard as families of $C^*$-algebras over $X$) $Cl_\tau(X)$ and $C_0(TX)$ with the graded suspension of $C_0(X)$.
See \cite{T5} for details.

\subsubsection{Analytic side}

In order to reformulate the localized index, we prepare several $KK$-elements for the situation of Section \ref{Non-compact manifolds}.
Recall that the pair $[S_X]$ and $[S_X^*]$ give a $KK$-equivalent between $C_0(X)$ and $Cl_\tau(X)$.
Similarly, $[S_{U_{2\vep}}]$, $[S_{U_{2\vep}}^*]$, $[S_M]$, $[S_M^*]$, $[S_{\nu_{\bb{C}}(M)}]$ and $[S_{\nu_{\bb{C}}(M)}^*]$ are used to reformulate the theorem in the following.


\begin{dfn}
$(1)$ $[k_*]':=[S_{U_{2\vep}}^*]\grotimes [k_*]\grotimes [S_X]\in KK_{S^1}(Cl_\tau(U_{2\vep}),Cl_\tau(X))$.
It is given by the zero-extension $Cl_\tau(U_\vep)\to Cl_\tau(X)$.


$(2)$ $[\tau_{U_{2\vep}}]':=[S_M^*]\grotimes [\tau_{U_{2\vep}}]\grotimes [S_{U_{2\vep}}]
\in KK_{S^1}(Cl_\tau(M ),Cl_\tau(U_{2\vep}))$.
It is given by
$${\bra{Cl_\tau(U_{2\vep}),
\pi,
\frac{C^{\rm fib}}{\sqrt{4\vep^2-\bra{C^{\rm fib}}^2}}}}.$$



$(3)$ $[e^{-1}_{U_{2\vep}}]':=[S_M^*]\grotimes [e^{-1}_{U_{2\vep}}]\grotimes [S_M]
\in KK_{S^1}(Cl_\tau(M ),Cl_\tau(M))$.
If $[e^{-1}_{U_{2\vep}}]$ is given by $\sum_{n}E_n\otimes  q^n$ under the isomorphism 
$\ca{R}KK_{S^1}(M ;C(M ),C(M ))_\pos \cong K^0(M )\otimes R(S^1)_\pos$, $[e^{-1}_{U_\vep}]'$ is given by
$$\sum_{n}\bra{E_n\grotimes \Cl_+(TM )}\otimes  q^n.$$
Note that $S^1$ acts on $\Cl_+(TM )$ trivially.


$(4)$ $[\ud{\bb{C}}_{M }]':=[\ud{\bb{C}}_{M }]\grotimes [S_M]\in KK_{S^1}(\bb{C},Cl_\tau(M ))$. It is given by
$$\bra{C(M ,S),1,0}.$$

$(5)$ $[L^*]':=[S_M^*]\grotimes [L^*]\grotimes [S_M]
\in KK_{S^1}(Cl_\tau(M ),Cl_\tau(M ))$. It is given by $[L^*]':=[L^*\grotimes \Cl_+(TM )]$.
\end{dfn}

\begin{dfn}
We define the following $KK$-elements:
$[\widetilde{k_*}]
:=\sigma_{\ca{S}_\vep}([k_*]')$,
$[\widetilde{\tau_{U_{2\vep}}}]
:=\sigma_{\ca{S}_\vep}([\tau_{U_{2\vep}}]')$,
$[\widetilde{e^{-1}_{U_{2\vep}}}]
:=\sigma_{\ca{S}_\vep}([e^{-1}_{U_{2\vep}}]')$ and
$[\widetilde{L^*}]
:=\sigma_{\ca{S}_\vep}([L^*]')$.
\end{dfn}

The following is clear from associativity of Kasparov product and the property of the homomorphism $\sigma_{\ca{S}_\vep}$.

The following construction will be generalized to Hilbert manifolds in the next section.

\begin{dfn}
We define $\ca{A}(X):=\ca{S}_\vep\grotimes Cl_\tau(X)$.
Similarly, we define $\ca{A}(U_{2\vep})$ and $\ca{A}(M)$.
\end{dfn}


\begin{lem}
The following diagram commutes.

$$\xymatrix{
KK_{S^1}(C_0(X),\bb{C})
 \ar^{[S_{X}^*]\grotimes-}[r]
 \ar_{[k_*]\grotimes -}[d] &
KK_{S^1}(Cl_\tau(X),\bb{C})
 \ar^{\sigma_{\ca{S}_\vep}}[r]
 \ar|{[k_*]'\grotimes -}[d] &
KK_{S^1}(\ca{A}(X),\ca{S}_\vep) 
 \ar^{[\widetilde{k_*}]\grotimes -}[d] \\
KK_{S^1}(C_0(U_{2\vep}),\bb{C})
 \ar^{[S_{U_{2\vep}}^*]\grotimes-}[r]
 \ar_{[\tau_{U_{2\vep}}]\grotimes -}[dd] &
KK_{S^1}(Cl_\tau(U_{2\vep}),\bb{C})
 \ar^{\sigma_{\ca{S}_\vep}}[r]
 \ar|{[\tau_{U_{2\vep}}]'\grotimes -}[d] &
KK_{S^1}(\ca{A}(U_{2\vep}),\ca{S}_\vep)
 \ar^{[\widetilde{\tau_{U_{2\vep}}}]\grotimes -}[d] \\
&
KK_{S^1}(Cl_\tau(U_{2\vep}),\bb{C})
 \ar^{\sigma_{\ca{S}_\vep}}[r]
 \ar|{[L^*]'\grotimes-}[d] &
KK_{S^1}(\ca{A}(U_{2\vep}),\ca{S}_\vep)
 \ar^{[\widetilde{L^*}]\grotimes-}[d] \\
KK_{S^1}(C(M ),\bb{C})
 \ar^{[S_{M }^*]\grotimes-}[r]
 \ar_{[e^{-1}_{U_{2\vep}}]\grotimes-}[d] &
KK_{S^1}(Cl_\tau(M ),\bb{C})
 \ar^{\sigma_{\ca{S}_\vep}}[r]
 \ar|{[e^{-1}_{U_{2\vep}}]'\grotimes-}[d] &
KK_{S^1}(\ca{A}(M ),\ca{S}_\vep) 
 \ar^{[\widetilde{e^{-1}_{U_{2\vep}}}]\grotimes-}[d] \\
KK_{S^1}(C(M ),\bb{C})_\pos
 \ar^{[S_{M }^*]\grotimes-}[r]
 \ar_{[\ud{\bb{C}}_{M }]\grotimes-}[d] &
KK_{S^1}(Cl_\tau(M ),\bb{C})_\pos
 \ar^{\sigma_{\ca{S}_\vep}}[r]
 \ar|{[\ud{\bb{C}}_{M }]'\grotimes-}[d] &
KK_{S^1}(\ca{A}(M ),\ca{S}_\vep)_\pos 
 \ar^{[\widetilde{\ud{\bb{C}}_{M }}]\grotimes-}[d]  \\
R(S^1)_\pos 
 \ar^=[r] &
R(S^1)_\pos
 \ar^{\sigma_{\ca{S}_\vep}}[r] &
KK_{S^1}(\ca{S}_\vep,\ca{S}_\vep)_\pos.}$$

\end{lem}
\begin{dfn}
We define the reformulated localized index 
$$\widetilde{\ind_{S^1}^\pos}:KK_{S^1}(\ca{A}(X),\ca{S}_\vep)\to KK_{S^1}(\ca{S}_\vep,\ca{S}_\vep)_\pos$$ 
by the composition of the right vertical arrows.
\end{dfn}

Then, the following is obvious.
\begin{thm}
For $[D]\in KK_{S^1}(C_0(X),\bb{C})$, we have
$$\sigma_{\ca{S}_\vep}(\ind_{S^1}^\pos([D]))=
\widetilde{\ind_{S^1}^\pos}(\sigma_{\ca{S}_\vep}([S_X^*]\grotimes [D])).$$

\end{thm}

\subsubsection{Topological side}

We have already deduced the fixed-point formula without using the $C_0(X)$-algebras $Cl_\tau(X)$ or $C_0(TX)$ in Section \ref{K-ori noncpt}.
By applying $\sigma_{\ca{S}_\vep}$ to each step of the construction of $(t-\ind_{S^1}^\pos)'$, we can reformulate the topological localized index.
Just like the previous subsection, for a homomorphism between $\ca{R}KK$-group $f'$, we define $\widetilde{f}$ so that the following diagram commutes:
$$\begin{CD}
\ca{R}KK(X;A_1,B_1) 
 @>\sigma_{\ca{S}_\vep}>> 
\ca{R}KK(X;\ca{S}_\vep\grotimes A_1,\ca{S}_\vep\grotimes B_1) \\
@Vf'VV @VV\widetilde{f}V \\
\ca{R}KK(X;A_2,B_2) 
 @>\sigma_{\ca{S}_\vep}>> 
\ca{R}KK(X;\ca{S}_\vep\grotimes A_2,\ca{S}_\vep\grotimes B_2).
\end{CD}$$

\begin{dfn}
We define $\widetilde{t-\ind_{S^1}^\pos}$ by the composition of the following homomorphisms:
$$\ca{R}KK_{S^1}(X;\ca{S}_\vep\grotimes C_0(X),\ca{S}_\vep\grotimes C_0(X)) 
 \xrightarrow{\widetilde{k^*}}
\ca{R}KK_{S^1}(U_\vep;\ca{S}_\vep\grotimes C_0(U_\vep),\ca{S}_\vep\grotimes C_0(U_\vep)) 
 \xrightarrow{\widetilde{j^*}} $$
$$\ca{R}KK_{S^1}(M ;\ca{S}_\vep\grotimes C(M ),\ca{S}_\vep\grotimes C(M )) 
 \xrightarrow{[\widetilde{L^*}]\grotimes-}
\ca{R}KK_{S^1}(M ;\ca{S}_\vep\grotimes C(M ),\ca{S}_\vep\grotimes C(M ))$$
$$\xrightarrow{[\widetilde{e^{-1}_{U_\vep}}]\grotimes-}
\ca{R}KK_{S^1}(M ;\ca{S}_\vep\grotimes C(M ),\ca{S}_\vep\grotimes C(M ))_\pos
 \xrightarrow{\widetilde{t-\ind}_{M }}
\bb{Q}\otimes KK_{S^1}(\ca{S}_\vep,\ca{S}_\vep)_\pos.$$
\end{dfn}

The following is obvious.
\begin{thm}
For $[u]\in \ca{R}KK_{S^1}(X;C_0(X),C_0(X))$, we have
$$\widetilde{t-\ind_{S^1}^\pos}(\sigma_{\ca{S}_\vep}([u]))
=\sigma_{\ca{S}_\vep}\bra{({t-\ind_{S^1}^\pos})'([u])}.$$
\end{thm}

\subsubsection{Poincar\'e duality}

We also reformulate the Poincar\'e duality homomorphism.
The reformulation of the local Bott element is given by the $KK$-equivalence $C_0(X)\cong Cl_\tau(X)$ and the $\ca{R}KK(X)$-equivalence $Cl_\tau(X)\cong C_0(TX)\cong C_0(X)$.

Strictly speaking, in order to describe the definition, we need a homomorphism 
$$\fgt:\ca{R}KK_{S^1}(X;\bullet,\bullet)\to KK_{S^1}(\bullet,\bullet)$$
defined by forgetting the $C_0(X)$-module structure.
See \cite[Section 3.2]{T5} for details.

\begin{dfn}
$(1)$ We define the {\bf reformulated local Bott element} $[\widetilde{\Theta_{X}}]$ by
$$[\widetilde{\Theta_{X}}]:=\sigma_{X,\ca{S}_\vep}\bra{[\Theta_{X}]\grotimes_{X}
\bbra{
\fgt([S_X])\grotimes [S_X^*]}}
\in \ca{R}KK_{S^1}(X;\ca{S}_\vep\grotimes C_0(X),\ca{A}(X)\grotimes \uwave{C_0(X)}).$$
The {\bf reformulated Poincar\'e duality} is defined by
$$\widetilde{\PD_X}:KK_{S^1}(\ca{A}(X),\ca{S}_\vep)
\ni [D]\mapsto
[\widetilde{\Theta_{X}}]\grotimes [D]\in
\ca{R}KK_{S^1}(X;\ca{S}_\vep\grotimes C_0(X),\ca{S}_\vep\grotimes C_0(X)).$$

$(2)$ We can similarly reformulate Definition \ref{PD for non complete}.
The reformulated version of $[\Theta_{U_\vep}']$ and $\PD_{U_\vep}'$ are denoted by
$[\widetilde{\Theta_{U_\vep}'}]$
and
$\widetilde{\PD_{U_{\vep}'}}$.
\end{dfn}

This $\ca{R}KK$-element has a quite simple representative.

\begin{dfn}
We define the {\bf local Bott homomorphism} $\beta_x:\ca{S}_\vep\to \ca{A}(X)=\ca{S}_\vep\grotimes Cl_\tau(X)$ by
$$\beta_x(f):=f(X\grotimes 1+1\grotimes C_x),$$
where $C_x$ is a vector field on $X$ given by
$$C_x(y):=\begin{cases}
-\log_{y}(x)\in T_yX & (y\in U_x), \\
\text{arbitrary vector of norm greater than }\vep & (y\notin U_x).
\end{cases}$$
Intuitively, $C_x(y)=$``$\overrightarrow{xy}$''.
\end{dfn}
\begin{rmk}
Although $C_x$ itself is not continuous, $f(X\grotimes 1+1\grotimes C_x)$ defines a continuous function since $f\in \ca{S}_\vep$ vanishes outside $(-\vep,\vep)$.
\end{rmk}

\begin{pro}[{\cite[Proposition 3.19]{T5}}]
$[\widetilde{\Theta_{X}}]$
is represented by the field of Kasparov modules
$$\bbra{\bra{\ca{A}(X),\beta_x,0}}_{x\in X}.$$
\end{pro}

By the same argument, we have the following.

\begin{cor}
$[\widetilde{\Theta_{U_\vep}'}]$
is represented by the field of Kasparov modules
$$\bbra{\bra{\ca{A}(X),\beta_x,0}}_{x\in U_\vep}.$$
\end{cor}

Then, we have the following fixed-point theorem.

\begin{thm}\label{non-cpt index thm}
The following diagram commutes:
$$\xymatrix{
KK_{S^1}(\ca{A}(X),\ca{S}_\vep) 
 \ar^-{\widetilde{\PD_X}}[r] \ar_-{[\widetilde{k_*}]\grotimes-}[d] &
\ca{R}KK_{S^1}(X;\ca{S}_\vep\grotimes C_0(X),\ca{S}_\vep\grotimes C_0(X))
 \ar^{\widetilde{k^*}}[dd] \\
KK_{S^1}(\ca{A}(U_{2\vep}),\ca{S}_\vep)
 \ar^-{\widetilde{\PD_{U_\vep}'}}[rd]
 \ar_-{[\widetilde{\tau_{U_{2\vep}}}]\grotimes-}[dd] & \\
 & 
\ca{R}KK_{S^1}(U_\vep;\ca{S}_\vep\grotimes C_0(U_\vep),\ca{S}_\vep\grotimes C_0(U_\vep)) 
 \ar^{\widetilde{j^*}}[d]  \\
KK_{S^1}(\ca{A}(M ),\ca{S}_\vep)
 \ar_-{\widetilde{\PD_{M }}}[r] 
 \ar_-{[\widetilde{e^{-1}_{U_{2\vep}}}]\grotimes-}[d] &
\ca{R}KK_{S^1}(M ;\ca{S}_\vep\grotimes C(M ),\ca{S}_\vep\grotimes C(M ))
 \ar^-{[\widetilde{e^{-1}_{U_{2\vep}}}]\grotimes-}[d] \\
KK_{S^1}(\ca{A}(M ),\ca{S}_\vep)
 \ar_-{\widetilde{\PD_{M }}}[r] 
 \ar_-{[\widetilde{L^*}]\grotimes-}[d] &
\ca{R}KK_{S^1}(M ;\ca{S}_\vep\grotimes C(M ),\ca{S}_\vep\grotimes C(M ))
 \ar^-{[\widetilde{L^*}]\grotimes-}[d] \\
KK_{S^1}(\ca{A}(M ),\ca{S}_\vep)_\pos
 \ar_-{\widetilde{\PD_{M }}}[r] 
 \ar_-{[\widetilde{\ud{\bb{C}}_{M }}]\grotimes -}[d] &
\ca{R}KK_{S^1}(M ;\ca{S}_\vep\grotimes C(M ),\ca{S}_\vep\grotimes C(M ))_\pos
 \ar^-{\widetilde{t-\ind_{M }}}[d] \\
KK_{S^1}(\ca{S}_\vep,\ca{S}_\vep)_\pos
 \ar@{^{(}->}^-{=}[r] &
\bb{Q}\otimes KK_{S^1}(\ca{S}_\vep,\ca{S}_\vep)_\pos.
}$$
%
\end{thm}
\begin{proof}
It suffices to prove that each square commutes. For example, commutativity of the first square can be checked as follows. Note that the following diagram commutes:
$$\begin{CD}
KK_{S^1}(Cl_\tau(X),\bb{C})
 @>[\Theta_X']\otimes->> 
\ca{R}KK_{S^1}(X; C_0(X),C_0(X)) \\
 @V[k'_*]\grotimes- VV @VVk^*V \\
KK_{S^1}(Cl_\tau(U_{2\vep}),\bb{C})
 @>[\Theta_{U_\vep}']'\otimes->> 
\ca{R}KK_{S^1}(U_\vep; C_0(U_{\vep}),C_0(U_{\vep})),
\end{CD}$$
where $[\Theta_{U_\vep}']'$ is defined by $[\Theta_{U_\vep}']$ with the $KK$-equivalence $C_0(U_{2\vep}) \cong Cl_\tau(U_{2\vep})$.

Since $\sigma_{\ca{S}_\vep}([x]\grotimes [y])=\sigma_{\ca{S}_\vep}([x])\grotimes \sigma_{\ca{S}_\vep}([y])$, the following extended diagram commutes:
{\footnotesize 
$$\xymatrix{
KK_{S^1}(\ca{A}(X),\ca{S}_\vep) 
 \ar^-{\widetilde{\PD_X}}[rrr] \ar_-{[\widetilde{k_*}]\grotimes-}[ddd] &&&
\ca{R}KK_{S^1}(X;\ca{S}_\vep\grotimes C_0(X),\ca{S}_\vep\grotimes C_0(X))
 \ar^{\widetilde{k^*}}[ddd] \\
&KK_{S^1}(Cl_\tau(X),\bb{C})
 \ar^{[\Theta_X']\otimes-}[r]
 \ar_{[k'_*]\grotimes-}[d] 
 \ar^{\sigma_{\ca{S}_\vep}}[lu] &
\ca{R}KK_{S^1}(X; C_0(X),C_0(X)) 
 \ar^{k^*}[d] 
 \ar_{\sigma_{\ca{S}_\vep}}[ru] & \\
&KK_{S^1}(Cl_\tau(U_{2\vep}),\bb{C})
 \ar^{[\Theta_{U_\vep}']'\otimes-}[r] 
 \ar_{\sigma_{\ca{S}_\vep}}[ld] &
\ca{R}KK_{S^1}(U_\vep; C_0(U_{\vep}),C_0(U_{\vep})) 
 \ar^{\sigma_{\ca{S}_\vep}}[rd] & \\
KK_{S^1}(\ca{A}(U_{2\vep}),\ca{S}_\vep)
 \ar^-{\widetilde{\PD_{U_\vep}'}}[rrr] &&& 
\ca{R}KK_{S^1}(U_\vep;\ca{S}_\vep\grotimes C_0(U_\vep),\ca{S}_\vep\grotimes C_0(U_\vep)). }
$$}
Commutativity of the outside square is what we wanted to prove. In fact, since each square commutes, so does the biggest one.

One can verify that the other square commute by parallel arguments.

\end{proof}

\begin{thm}
For any $[D]\in KK_{S^1}(\ca{A}(X),\ca{S}_\vep)$, 
$$\widetilde{\ind_{S^1}^\pos}([D])=
\widetilde{t-\ind_{S^1}^\pos}(\widetilde{\PD_X}([D])).$$
\end{thm}

This is the correct form of the localized index theorem to formulate an infinite-dimensional version.

\subsection{Comparison with Hochs-Wang's index}\label{Comparison with the Hochs-Wang's index}

Peter Hochs and Hang Wang constructed an index for non-compact manifolds equipped with a torus action with compact fixed-point set in \cite{HW}. 
In this subsection, we compare it with our index.

We begin with a review of the construction of \cite{HW} for $G=S^1$. 

\begin{dfn}
For each $x\in S^1$, we have a ring homomorphism $ch_x:R({S^1})\to \bb{C}$ given by $ch_x( \rho):={\rm tr}( \rho(x))$ for a virtual representation $ \rho$ of ${S^1}$.
For $g\in {S^1}$, let $I_g:=\ker(ch_g)$.
The {\bf localization at $g$} is the localization in the sense of commutative rings: $(R({S^1})-I_g)^{-1}R({S^1})$. It is denoted by $R({S^1})_g$. Note that it is a unital $R({S^1})$-algebra.

Associated to it, we introduce the following symbols:
\begin{itemize}
\item For an $R({S^1})$-module $\ca{M}$, we define $\ca{M}_g:=\ca{M}\otimes_{R({S^1})}R({S^1})_g$.

\item The natural homomorphism $\ca{M}\ni m\mapsto m\otimes 1\in \ca{M}\otimes_{R({S^1})}R({S^1})_g=\ca{M}_g$ is denoted by $\loc_g$.

\item For $R({S^1})$-modules $\ca{M}$ and $\ca{N}$ and an $R({S^1})$-module homomorphism $F:\ca{M}\to \ca{N}$, the induced homomorphism between the localizations $F\otimes\id_{R({S^1})_g}:\ca{M}_g\to \ca{N}_g$ is denoted by $F_g$.
\end{itemize}
\end{dfn}

Let $X$ be a finite-dimensional complete Riemannian manifold equipped with an isometric action of  ${S^1}$. Let $g\in {S^1}$ be a generator. Suppose that the fixed-point set $M=X^g$ is compact. Take relatively compact ${S^1}$-invariant neighborhoods $U$ and $V$ of $M $ satisfying that $\overline{V}\subseteq U$. We denote the inclusions by
$$M \xrightarrow{j}V\xrightarrow{i}U \xrightarrow{k} X.$$
Since $U$ is relatively compact, the following result holds.

\begin{pro}[{\cite[Theorem 2.3]{HW}}]
$i_*:K_0^{{S^1}}(\overline{V})\to K_0^{{S^1}}(U)$ is invertible after localization at $g$.
\end{pro}

\begin{dfn}[{\cite[Section 2.2]{HW}}]
We define the {\bf localized index at $g$} by the composition of the homomorphisms
$$KK_{S^1}(C_0(X),\bb{C})\xrightarrow{[k_*]\grotimes -}KK_{S^1}(C_0(U),\bb{C})\xrightarrow{\loc_g}KK_{S^1}(C_0(U),\bb{C})_g\xrightarrow{(i^*)_g^{-1}}$$
$$KK_{S^1}(C(\overline{V}),\bb{C})_g\xrightarrow{[\bb{C}_{\overline{V}}]\grotimes-}KK_{S^1}(\bb{C},\bb{C})_g=R({S^1})_g.$$
The localized index at $g$ is denoted by $\ind_g:K_0^{S^1}(X)\to R({S^1})_g$.
\end{dfn}

We reformulate it by using a more ``delicate'' localization.

\begin{dfn}
We identify $R(G)$ with the ring of characters.
Let $S_g$ be the multiplicative closed set consisting of $f(z)=\sum_{n\geq n_0} a_nz^n$ such that $f(g)\neq 0$ and $a_{n_0}=\pm 1$. 
We define $R({S^1})_{g,\New}:=S_g^{-1}R({S^1})$.

Associated to it, we introduce the following symbols:
\begin{itemize}
\item For an $R({S^1})$-module $\ca{M}$, we define $\ca{M}_{g,\New}=\ca{M}\otimes_{R({S^1})}R({S^1})_{g,\New}$.
\item The homomorphism $\ca{M}\ni m\mapsto m\otimes 1\in \ca{M}\otimes_{R({S^1})}R({S^1})_{g,\New}=\ca{M}_{g,\New}$ is denoted by $\loc_{g,\New}$.
\item For $R({S^1})$-modules $\ca{M}$ and $\ca{N}$ and an $R({S^1})$-module homomorphism $F:\ca{M}\to \ca{N}$, $F\otimes\id_{R({S^1})_{g,\New}}:\ca{M}_{g,\New}\to \ca{N}_{g,\New}$ is denoted by $F_{g,\New}$.
%
\end{itemize}
\end{dfn}

Three constructions $R(S^1)_\pos$, $R(S^1)_g$ and $R(S^1)_{g,\New}$ are related as follows.

\begin{pro}
We can define natural injective ring homomorphisms
$$\Phi_1:R(S^1)_{g,\New}\to R(S^1)_{g},\ \ \ 
\Phi_2:R(S^1)_{g,\New}\to R(S^1)_\pos$$
so that the restrictions to $R(S^1)$ are identity maps.
\end{pro}
\begin{proof}
Thanks to the universal property of localization, it suffices to see that each element of $S_g$ is mapped to an invertible element of $R(S^1)_{g}$ and $R(S^1)_\pos$. The former is obvious. The latter is clear from Lemma \ref{fps invertible}.
\end{proof}

The following is clear from the same construction of Lemma \ref{Loc thm fps} in Appendix.

\begin{lem}
$[i^*]$ is invertible after tensoring with $R({S^1})_{g,\New}$.
\end{lem}

By replacing $[i^*]_g^{-1}$ with $[i^*]_{g,\New}^{-1}$, we can define a new $g$-index $\ind_{g,\New}:KK_{S^1}(C_0(X),\bb{C})\to R({S^1})_{g,\New}$. With the same argument of \cite{HW}, we can check that $\ind_{g,\loc}$ is independent of the choice of $U$ and $V$. 
In the following proposition, we adopt $U_{2\vep}$ as $U$.



Let us compare three index homomorphisms: $\ind_{S^1}^\pos$, $\ind_g$ and $\ind_{g,\New}$.

\begin{pro}
The following diagram commutes.
$$\xymatrix{
& K_{S^1}(X) \ar_-{\ind_{S^1}^\pos}[dl] \ar|-{\ind_{g,\New}}[d] \ar^{\ind_g}[dr] & \\
R(S^1)_\pos  &
R(S^1)_{g,\New} \ar^-{\Phi_2}[l] \ar_-{\Phi_1}[r] &
R(S^1)_{g}.}$$
\end{pro}
\begin{proof}
It is clear from the following commutative diagram.
{\footnotesize
$$\xymatrix{
&KK_{S^1}(C_0(X),\bb{C}) 
 \ar@{=}[r] 
 \ar_{[k_*]\grotimes-}[d] &
KK_{S^1}(C_0(X),\bb{C})
 \ar^{[k_*]\grotimes-}[d] & & \\
KK_{S^1}(C_0(U),\bb{C}) 
 \ar@{=}[r] 
 \ar_{[\tau_{U}]\grotimes-}[d] &
KK_{S^1}(C_0(U),\bb{C}) 
 \ar@{=}[r] &
KK_{S^1}(C_0(U),\bb{C}) 
 \ar^-{\loc_{g,\New}}[r] &
KK_{S^1}(C_0(U),\bb{C})_{g,\New}
 \ar^{([i^*]_{g,\New})^{-1}\grotimes-}[dd] 
 \ar^-{\id\otimes \Phi_1}[r] &
KK_{S^1}(C_0(U),\bb{C})_{g}
 \ar^{([i^*]_{g})^{-1}\grotimes-}[dd] \\
KK_{S^1}(C(X^g),\bb{C}) 
 \ar_{[e_{U}^{-1}]\grotimes-}[d] & & &\\
KK_{S^1}(C(X^g),\bb{C})_\pos 
 \ar_{[\ud{\bb{C}}_{X^g}]\grotimes-}[d] &
KK_{S^1}(C(X^g),\bb{C}) 
 \ar^{f,\cong}[r] 
 \ar^{[j^*]\grotimes-}[uu]
 \ar^{[\ud{\bb{C}}_{X^g}]\grotimes-}[d] 
 \ar_{\pos}[l] &
KK_{S^1}(C(\overline{V}),\bb{C}) 
 \ar_{[i^*]\grotimes-}[uu] 
 \ar^{[\ud{\bb{C}}_{\overline{V}}]\grotimes-}[d] 
 \ar^-{\loc_{g,\New}}[r] &
KK_{S^1}(C(\overline{V}),\bb{C})_{g,\New}
 \ar^{[\ud{\bb{C}}_{\overline{V}}]\grotimes-}[d]
 \ar^{\id\otimes \Phi_1}[r] 
 \ar_{f^{-1}\otimes \Phi_2}@/_27pt/[lll]&
KK_{S^1}(C(\overline{V}),\bb{C})_{g,\New}
 \ar^{[\ud{\bb{C}}_{\overline{V}}]\grotimes-}[d] \\
R(S^1)_\pos &
R(S^1) 
 \ar@{=}[r] 
 \ar_{\pos}[l]&
R(S^1) 
 \ar^-{\loc_{g,\New}}[r] &
R(S^1)_{g,\New}
 \ar_{\Phi_1}[r]
 \ar^{\Phi_2}@/^27pt/[lll] &
R(S^1)_{g},}$$}
where $f:KK_{S^1}(C(X^g),\bb{C})\to KK_{S^1}(C(\overline{V}),\bb{C})$ is induced by the homotopy equivalent.
\end{proof}

In this sense, our index is almost the same with that of \cite{HW}. However, our framework is more flexible in the point that our $KK$-theory admits infinite-rank vector bundles if it is of ``positive energy''. This flexibility enables us to deal with loop spaces.

\section{An $S^1$-equivariant index theorem for loop spaces}\label{Main section}


We want to do the same thing of the previous section on loop spaces.
The phrase ``fixed-point formula for loop spaces'' reminds us of the Witten genus \cite{Wit}.
Unfortunately, we have not proved that our index realizes it. Including this point, in the final subsection, we will summarize remained problems.

\subsection{Loop spaces}

For a compact Riemannian manifold $M$, the smooth loop space of $M$ is denoted by $C^\infty(S^1,M)$.
In the present paper, we consider a completion of it.
Since the tangent space at $\gamma\in C^\infty(S^1,M)$ can be identified with the section space $C^\infty(S^1,\gamma^*TM)$, we can define a family of Riemannian metrics as follows.
Let $g$ be the Riemannian metric on $M$.

\begin{dfn}
The $L^2_s$-metric on $C^\infty(S^1,M)$ is defined by the inner product
$$\inpr{X}{Y}{s}:=\frac{1}{2\pi}\int_{S^1} g_{\gamma(\theta)}\bra{\bra{\id{-\frac{d^2 }{d\theta^2}}}^sX(\theta),Y(\theta)}d\theta.$$
\end{dfn}

This metric gives a metric space structure on $C^\infty(S^1,M)$.
The completion of it is the main subject of this section.

\begin{dfn}
The completion of $C^\infty(S^1,M)$ with respect to the $L^2_s$-metric is denoted by $LM_{L^2_s}$.
\end{dfn}
\begin{rmk}
For simplicity, we will concentrate on the case when $s=0$, and we denote $LM_{L^2_0}$ by $LM$.

\end{rmk}

In this case, the following has been established.
Let $\ev_\theta:C^\infty(S^1,M)\to M$ be the evaluation map at $\theta\in S^1$ and let $\nabla^{\LC}$ be the Levi-Civita connection on $M$.
A vector field $X$ on $LM$ is said to be ``taking values in smooth vector fields'' if $X(\gamma)\in L^2(S^1,\gamma^*TM)$ is smooth for each $\gamma\in C^\infty(S^1,M)$.

\begin{thm}[{\cite[Lemma 2.1]{MRT}}]
For vector fields $X,Y$ on $LM$ taking values in smooth vector fields and $\gamma\in C^\infty (S^1,M)$, we define a vector field $D_XY$ by
$$D_XY(\gamma,\theta):=(\ev_\theta^*\nabla^{\LC})_{X}\bbra{Y(\gamma,\theta)}.$$
Then, the correspondence $(X,Y)\mapsto D_XY$ is the Levi-Civita connection of $LM$.
\end{thm}

Recall that the curvature of a connection induced by a smooth map from another connection is given by the pullback of the curvature of the original one. Thus, the curvature operator on $LM$ is given by
$$R(X,Y)Z(\gamma,\theta)=R^M(X(\gamma,\theta),Y(\gamma,\theta))Z(\gamma,\theta),$$
where $R^M$ is the curvature tensor of $M$ with respect to the Levi-Civita connection.
In particular, we have the following result.

\begin{pro}
Let $\delta$ be the maximum value of the absolute value of sectional curvatures of $M$.
For each $\gamma\in LM$, an orthonormal two-frame $\{u,v\}$ of $T_\gamma LM$, we have
$$|\inpr{R(u,v)v}{u}{}|\leq \delta.$$
\end{pro}
\begin{proof}
It is clear from $\inpr{R(u,v)v}{u}{}
=\frac{1}{2\pi}\int_{S^1}g_{\gamma(\theta)}\bra{R^M(u(\theta),v(\theta))v(\theta),u(\theta)}d\theta.$
\end{proof}



\subsection{Ingredients for a loop space version of the index theorem}

In Section \ref{reform of fpt}, we reformulated the fixed-point formula for $\ind_{S^1}^\pos$.
In this subsection, we prepare an infinite-dimensional versions of the ingredients of this theorem.

\subsubsection{$C^*$-algebras of Hilbert manifolds and the zero-extension}


We begin with the definition of the ``function algebras for Hilbert manifolds'', which were introduced in \cite{Yu} as  generalizations of the $C^*$-algebras of infinite-dimensional Hilbert-Hadamard spaces \cite{GWY}.
Since detailed properties are studied in \cite{T5}, we just define it and explain necessary properties.

In this subsection we deal with a Hilbert manifold $\ca{X}$ satisfying the following assumption.


\begin{asm}\label{Katei}
Let $\ca{X}$ be a Hilbert manifold whose all sectional curvatures are bounded above by $\delta$ and the injectivity radius is greater than $2\vep>0$ at each point.
When $\delta>0$, we assume that $\vep<{\pi}/{2\sqrt{\delta}}$ from the beginning by re-taking $\vep$ smaller if necessary.
\end{asm}

\begin{dfn}[{\cite[Definition 5.1]{GWY}}]\label{def field of Clifford algebras}
We consider the space
$$\Pi(\ca{X}):=\prod_{(x,t)\in \ca{X}\times [0,\vep)}\Cl_+(T_x\ca{X}\oplus t\bb{R}),$$
where 
$$t\bb{R}:=\begin{cases}
\bb{R} & (t\neq 0) \\
0 & (t=0).\end{cases}
$$
This is a space of possibly non-continuous Clifford algebra-valued functions.
Then we consider a huge $C^*$-algebra
$$\Pi_b(\ca{X}):=\bbra{s\in \prod(\ca{X})\ \middle|\  \|s(x,t)\|\text{ is bounded.}}$$
equipped with the pointwise algebraic operations (addition, multiplication and the adjoint) and the uniform norm.
\end{dfn}

The following definition is parallel to the $C^*$-algebra for a Hilbert-Hadamard space \cite[Definition 5.14]{GWY}.

\begin{dfn}[\cite{Yu}]
$(1)$ Let $\ca{X}$ be a Hilbert manifold satisfying Assumption \ref{Katei}. 
Let $x_0,x\in \ca{X}$, and suppose that $d(x,x_0)<2\vep$. Then $x_0$ is contained in the image of $\exp_{x}:B_{2\vep}(T_{x}\ca{X})\to \ca{X}$, and hence it is contained in the domain of $\log_x:\exp_x(B_{2\vep}(T_{x}\ca{X}))\to T_{x}\ca{X}$.
The local Clifford operator at $x_0$ is defined by
$$C_{x_0}(x,t):=
(-\log_{x}(x_0),t)\in T_x\ca{X}\oplus t\bb{R},$$
or equivalently $C_{x_0}(x,t)=\bra{(-d\exp_{x_0})_x(\log_{x_0}(x)),t}$, or more intuitively ``$C_{x_0}(x,t)=(\overrightarrow{x_0x},t)$''.

$(2)$ The Bott homomorphism $\beta_{x_0}:\ca{S}_\vep\to \Pi_b(\ca{X})$ centered at $x_0\in \ca{X}$ is defined by the following: For $f\in \ca{S}_\vep$,
$$\beta_{x_0}(f)(x,t):=
\begin{cases}
f(C_{x_0}(x,t)) & (d(x,x_0)<\vep) \\
0 & (d(x,x_0)\geq\vep),
\end{cases}$$
where $f(C_{x_0}(x,t))$ is the functional calculus in the $C^*$-algebra $\Cl_+(T_x\ca{X}\oplus t\bb{R})$.

$(3)$ The $C^*$-algebra $\ca{A(X)}$ is defined by the $C^*$-subalgebra of $\Pi_b(\ca{X})$ generated by the image of the Bott homomorphisms:
$$\ca{A(X)}:=C^*\bra{\bbra{
\beta_{x_0}(f)\ \middle|\  x_0\in\ca{X}, f\in \ca{S}_\vep}}.$$

$(4)$ For a subset $\ca{Y}$ of $\ca{X}$, we define $\ca{A}(\ca{X},\ca{Y})$ by the $C^*$-subalgebra of $\ca{A}(\ca{X})$ generated by $\bbra{\beta_x(f)\, \middle|\, x\in \ca{Y},f\in \ca{S}_\vep}$.
\end{dfn}

It shares several properties with the $C^*$-algebra of \cite{GWY}.





\begin{pro}[{\cite[5.15, 7.1, 7.2]{GWY}, \cite[5.9, 5.10]{T5}}]\label{properties of GWY algebra}
Let $\ca{X}$ be a Hilbert manifold satisfying Assumption \ref{Katei}.


$(1)$ $\ca{A(X)}$ is separable whenever $\ca{X}$ is separable.



$(2)$ The group of all isometries of $\ca{X}$ continuously act on $\ca{A}(\ca{X})$.
\end{pro}

\begin{dfn}
Let $\ca{X}$ be a Hilbert manifold satisfying Assumption \ref{Katei} and let $\ca{Y}$ be a subset of $\ca{X}$.
For the natural inclusion $k:\ca{Y}\hookrightarrow\ca{X}$, we define the natural injective $*$-homomorphism $\ca{A}(\ca{X},\ca{Y})\to \ca{A(X)}$, and it is denoted by $\widetilde{k_*}$.
\end{dfn}

The $C^*$-algebra $\ca{A(X,Y)}$ plays a role of ``$\ca{A}$($\vep$-neighborhood of $\ca{Y}$)''.

\subsubsection{A Thom homomorphism-like construction}

We now prove one more functorial property.

\begin{pro}
Let $\ca{X}$ be a Hilbert manifold satisfying Assumption \ref{Katei}. 
Let $M$ be a totally geodesic submanifold of $\ca{X}$.
The inclusion is denoted by $i:M\hookrightarrow \ca{X}$.
Then, there exists a $C^*$-algebra homomorphism 
$\Pi_b(M)\to \Pi_b(\ca{X})$ so that the following diagram commutes for arbitrary $p\in M$:
$$\xymatrix{
& 
\ca{S}_\vep
 \ar^-{\beta_{i(p)}}[rd] 
 \ar_-{\beta_{p}}[ld] & \\
\Pi_b(M) \ar[rr]&& \Pi_b(\ca{X}).}$$
\end{pro}
\begin{proof}
Note that $i_*:TM\to T\ca{X}|_{M}$ is  isometric.
Thus, it extends to an injective fiberwise $*$-homomorphism
$i_*:\Cl_+(TM)\to \Cl_+(T\ca{X}|_{M})$, and hence we can define $i_*:\Pi_b(M)\to \Pi_b(\ca{X})$.
Since $M$ is totally geodesic in $\ca{X}$, we have 
$$i(\exp_{x_0}(v))=\exp_{i(x_0)}(i_*(v)),$$
$$i_*\bbra{\log_{x_0}(x)}=\log_{i(x_0)}(i(x)),$$
$$i_*(C_{x_0}(x,t))=C_{i(x_0)}(i(x),t).$$
Therefore, the above triangle commutes.
\end{proof}

Now it is clear that the following diagram commutes for arbitrary $p\in M$:
$$\xymatrix{
\ca{A}(M) 
 \ar@{^{(}->}[dd] \ar@{.>}[rr]&&
\ca{A(X)} 
 \ar@{^{(}->}[dd] \\
& 
\ca{S}_\vep
 \ar|-{\beta_{i(p)}}[ru] 
 \ar|-{\beta_{i(p)}}[rd] 
 \ar|-{\beta_{p}}[lu] 
 \ar|-{\beta_{p}}[ld] & \\
\Pi_b(M) \ar_{i_*}[rr]&& \Pi_b(\ca{X}).}$$
Thus, the above dotted arrow makes sense.

\begin{cor}\label{cor tot geod and A(X)}
In the same situation of the previous proposition, there exists a $C^*$-algebra homomorphism $\ca{A}(M)\to \ca{A}(\ca{X})$ satisfying that
$$\ca{A}(M)\ni \beta_p(f)\mapsto \beta_{i(p)}(f)\in \ca{A}(\ca{X}).$$
\end{cor}

We want to apply this construction to loop spaces.
For this aim, we need the following.

\begin{thm}[{\cite[Chapter II Theorem 5.1]{Kob}}]
For a Riemannian manifold $X$ equipped with an isometric group action of a group $G$, the fixed-point set $X^G$ is totally geodesic.
\end{thm}

One can easily see that this result holds for an infinite-dimensional Hilbert manifold.

Since the set of constant loops in $LM$ is the fixed-point set with respect to the $S^1$-action, we have the following.

\begin{cor}
$M$ is totally geodesic in $LM$ with respect to arbitrary $S^1$-invariant metric.
\end{cor}





Now, it is clear that the following makes sense.
Let $\nu(M)$ be the normal bundle of $M$ in $LM$, $\nu(M)_\delta:=\bbra{v\in \nu(M)\midd \|v\|<\delta}$ and $U_\delta:=\exp^\perp(\nu(M)_\delta)$.

\begin{dfn}\label{Thom for LM}
We define a $*$-homomorphism 
$\widetilde{\tau_{U_{2\vep}}}:\ca{A}(M)\to \ca{A}(LM,U_{2\vep})$ by Corollary \ref{cor tot geod and A(X)}. We call it the {\bf Thom homomorphism}.
\end{dfn}



\begin{rmk}

Let us explain the reason why we call it the Thom homomorphism.

In a finite-dimensional setting, we can introduce the following.
Let $X$ be a complete manifold and let $M$ be a totally geodesic submanifold.
Let $U_\vep$ be the $\vep$-neighborhood of $M$ in $X$ and let $\varpi:U_\vep\to M$ be the orthogonal projection.
We can define $\tau_{U_\vep}:\ca{A}(M)\to \ca{A}(U_\vep)$ such that 
$\beta_p(f)\mapsto \beta_{i(p)}(f)$.
We define a fiberwise Clifford operator $C^{\rm fib}_x$ for each $x\in M$ as a local vector filed on $M$ by
$$C^{\rm fib}_x(u,t):=(\exp^\perp_{x*}(\log^\perp_x(u)),t)=(\text{``}\overrightarrow{xu}\text{''},t)\in T_u^{\rm fib}X\oplus t\bb{R}$$
for $u\in U_\vep$ so that $\varpi(u)=x$.
Then, for $f\grotimes h\in \ca{S}_\vep\grotimes Cl_\tau(M)\cong \ca{A}(M)$, $\widetilde{\tau_{U_\vep}}(f\grotimes h)(u,t)$ is given by
$$f(X\grotimes 1+1\grotimes C_{\varpi(u)})(t,u)\grotimes h(\varpi(u))\in
\Cl_+(t\bb{R})\grotimes \Cl_+(T_{u}^{\rm fib}X)\grotimes \Cl_+(T_{\varpi(u)}M).$$
This is because the Clifford operator of $U_\vep$ is given by the sum of the  Clifford operator of $M$ and the fiberwise Clifford operator.

Since the correspondence $f\mapsto f(X\grotimes 1+1\grotimes C_{\varpi(u)})$ induces the Bott periodicity homomorphism of \cite{HKT}, the family version
$$f\grotimes h\mapsto f(X\grotimes 1+1\grotimes C)\grotimes h$$
induces the Thom isomorphism.
Definition \ref{Thom for LM} is obviously an infinite-dimensional version of this construction.
\end{rmk}



\subsubsection{Inverse Euler class}

Since the restriction homomorphism has not been constructed and the bundle $\nu(M)$ is of infinite-rank, we can not follow the same story in order to define the Euler class.
In our case, however, we can directly define the Euler class as follows, thanks to the complex structure of the normal bundle $\nu(M)$.

\begin{lem}
$(1)$ The restriction of the tangent bundle of $LM$ to $M$ is given by
$$TLM|_M=\coprod_{m\in M} L^2(S^1,T_mM).$$
By the Fourier series theory, 
$$L^2(S^1,T_mM)\cong T_mM\oplus \bra{ T_mM\otimes \bigoplus_{k>0}\bbra{\bb{R}\cos(k\theta)\oplus \bb{R}\sin(k\theta)}}.$$
Thus, the normal bundle is given by
$$\nu(M)\cong \coprod_{m\in M} T_mM\otimes \bigoplus_{k>0}\bbra{\bb{R}\cos(k\theta)\oplus \bb{R}\sin(k\theta)}.$$

$(2)$ By the complex structure $J(\cos(k\theta)):=-\sin(k\theta)$ and $J(\sin(k\theta))=\cos(k\theta)$,
$$\nu_\bb{C}(M)\cong
\bigoplus_{k>0}(T_mM\otimes \bb{C})\otimes \bb{C}_k$$
as complex vector bundles,
where $\bb{C}_k$ is the representation space of $S^1$ with weight $k$.
As usual, when we regard $\nu(M)$ as a complex vector bundle, we denote it by $\nu_\bb{C}(M)$.
\end{lem}

Although it is of infinite-rank, it does define an element of $K^0_{S^1}(M)_\pos$.
Moreover, the exterior product of it also makes sense in $K^0_{S^1}(M)_\pos$, and it is invertible.


\begin{pro}
$(1)$ The exterior tensor product of $\nu_{\bb{C}}(M)$,
$$\bigwedge^*\nu_\bb{C}(M)=\bigwedge^*\bigoplus_{k>0}(TM\otimes \bb{C})\otimes  \bb{C}z^k$$
defines an element of  $ KK_{S^1}(C(M),C(M))_\pos$.
The corresponding $KK$-element is denoted by $[e_{U_{2\vep}}]$.

$(2)$ It is invertible.
\end{pro}
\begin{proof}
When we work on $K$-theory, we denote the complexification of the tangent bundle $TM\otimes \bb{C}$ by $T$.

$(1)$ We compute the exterior product.
\begin{align*}
\bigwedge^*\sum_{k>0} T \otimes  q^k 
&=\prod_{k>0}\bigwedge^* T \otimes  q^k \\
&=\prod_{k>0}\Bigl\{
\ud{\bb{C}}_M\otimes  q^0-
 T \otimes  q^k+
\wedge^2 T \otimes  q^{2k}-
\wedge^3 T \otimes  q^{3k} \Bigr. \\
&\ \ \ \ \ \ \Bigl.+\cdots+(-1)^{\dim(M)}
\wedge^{\dim(M)} T \otimes  q^{{\dim(M)}k} \Bigr\} \\
&=\ud{\bb{C}}_M\otimes  q^0
+\bra{- T }\otimes  q^1
+\bra{\wedge^2 T - T }\otimes  q^2 +\bra{-\wedge^3 T + T\otimes T - T }\otimes  q^3 \\
&\ \ \ \ \ \ +\bra{\wedge^4 T +T\otimes \wedge^2 T+T\otimes T+\wedge^2 T - T }\otimes  q^4
+\cdots.
\end{align*}
It makes sense in $KK_{S^1}(C(M),C(M))_\pos$ since each coefficient of $ q^k$ is a finite-rank vector bundle.

$(2)$ Since the coefficient of $ q^0$ is the rank one trivial bundle, from Lemma \ref{fps invertible}, $[e_{U_{2\vep}}]$ is invertible.
\end{proof}

\begin{dfn}\label{inv Eul for LM}
$(1)$ We define $[e^{-1}_{U_{2\vep}}]\in KK_{S^1}(C(M),C(M))_\pos$ by the inverse of $[e_{U_{2\vep}}]$.

$(2)$ The reformulated version of it is denoted by
$$[\widetilde{e^{-1}_{U_{2\vep}}}]
:=\sigma_{\ca{S}_\vep}([S_M^*]\grotimes [e^{-1}_{U_{2\vep}}]\grotimes [S_M])
\in KK_{S^1}(\ca{A}(M),\ca{A}(M))_\pos.$$
\end{dfn}

\subsection{The index and the fixed-point formula}


Now we can formulate a loop space version of the localized index and the fixed-point formula.

\begin{dfn}
We define a group homomorphism $\widetilde{\ind_{S^1}^\pos}: KK_{S^1}(\ca{A}(LM),\ca{S}_\vep)\to KK_{S^1}(\ca{S}_\vep,\ca{S}_\vep)_\pos$ by the composition
$$KK_{S^1}(\ca{A}(LM),\ca{S}_\vep)
\xrightarrow{[\widetilde{k_*}]\grotimes -}
KK_{S^1}(\ca{A}(LM,U_{2\vep}),\ca{S}_\vep)
\xrightarrow{[\widetilde{\tau_{U_{2\vep}}}]\grotimes -}
KK_{S^1}(\ca{A}(M),\ca{S}_\vep)
\xrightarrow{[\widetilde{e^{-1}_{U_{2\vep}}}]\grotimes-}$$
$$KK_{S^1}(\ca{A}(M),\ca{S}_\vep)_\pos
\xrightarrow{[\widetilde{\ud{\bb{C}}_{M}}]\grotimes-}
KK_{S^1}(\ca{S}_\vep,\ca{S}_\vep)_\pos.$$
We call this homomorphism the {\bf localized index for $LM$}. 
\end{dfn}

\begin{rmk}
In section \ref{reform of fpt}, we reformulated the localized index for $K$-oriented manifolds using the difference line bundle $[L]$ specified by the Spinor bundles.
However, the diagram in Theorem \ref{non-cpt index thm} commutes for arbitrary line bundle $L$. 
{\it For simplicity, we ignored $L$ to formulate a loop space version of the index.}
Although to study what line bundle is appropriate for $L$ in our situation seems to be an interesting problem, we do not study this problem further.
\end{rmk}

In this subsection, by translating each step of the construction of the localized index into the topological language, we deduce a cohomology formula of the localized index for $LM$.
For an $S^1$-equivariant continuous map $\phi:\ca{X}\to \ca{Y}$, we can define a pullback homomorphism 
$$\phi^*:\ca{R}KK_{S^1}(\ca{Y},\scr{C}(\ca{Y}),\scr{C}(\ca{Y}))\to \ca{R}KK_{S^1}(\ca{X},\scr{C}(\ca{X}),\scr{C}(\ca{X})).$$
The corresponding homomorphism
$$\ca{R}KK_{S^1}(\ca{Y},\ca{S}_\vep\grotimes \scr{C}(\ca{Y}),\ca{S}_\vep\grotimes \scr{C}(\ca{Y}))\to 
\ca{R}KK_{S^1}(\ca{X},\ca{S}_\vep\grotimes \scr{C}(\ca{X}),\ca{S}_\vep\grotimes \scr{C}(\ca{X}))$$
is denoted by $\widetilde{\phi^*}$.


\begin{thm}
The following diagram commutes.
$$\xymatrix{
KK_{S^1}(\ca{A}(LM),\ca{S}_\vep) 
 \ar^-{\widetilde{\PD_{LM}}}[r]
 \ar_-{[\widetilde{k_*}]\grotimes-}[d]
 \ar@{}|{(1)}[rd]
  &
\ca{R}KK_{S^1}(LM;\ca{S}_\vep\grotimes \scr{C}(LM),\ca{S}_\vep\grotimes \scr{C}(LM))
 \ar^{\widetilde{k^*}}[dd] 
 \\
KK_{S^1}(\ca{A}(LM,U_{2\vep}),\ca{S}_\vep)
 \ar^-{\widetilde{\PD_{U_\vep}'}}[rd]
 \ar_-{[\widetilde{\tau_{U_{2\vep}}}]\grotimes-}[dd] 
 \ar@{}|{(2)}[rddd]
 &
  \\
 & 
\ca{R}KK_{S^1}(U_\vep;\ca{S}_\vep\grotimes \scr{C}(U_\vep),\ca{S}_\vep\grotimes \scr{C}(U_\vep)) 
 \ar^{\widetilde{i^*}}[d] 
\\
KK_{S^1}(\ca{A}(M),\ca{S}_\vep)
 \ar_-{\widetilde{\PD_{M}}}[r] 
 \ar_-{[\widetilde{e^{-1}_{U_{2\vep}}}]\grotimes-}[d]
 \ar@{}|{(3)}[rd] 
 &
\ca{R}KK_{S^1}(M;\ca{S}_\vep\grotimes C(M),\ca{S}_\vep\grotimes C(M))
 \ar^-{[\widetilde{e^{-1}_{U_{2\vep}}}]\grotimes-}[d]
  \\
KK_{S^1}(\ca{A}(M),\ca{S}_\vep)_\pos
 \ar_-{\widetilde{\PD_{M}}}[r] 
 \ar_-{[\widetilde{\ud{\bb{C}}_{M}}]\grotimes -}[d] 
  \ar@{}|{(4)}[rd] 
&
\ca{R}KK_{S^1}(M;\ca{S}_\vep\grotimes C(M),\ca{S}_\vep\grotimes C(M))_\pos
 \ar^-{\widetilde{t-\ind_{M}}}[d] \\
R(S^1)_\pos
 \ar_{=}[r]&
R(S^1)_\pos .
}$$
We call the composition of the right vertical arrows the {\bf topological localized index for $LM$} and we denote it by $\widetilde{t-\ind_{S^1}^\pos}$.
\end{thm}
\begin{proof}
$(1)$ Let $[D]=(E,\pi,D)\in KK_{S^1}(\ca{A}(LM),\ca{S}_\vep)$.
Then $[\widetilde{k_*}]\grotimes[D]=(E,\pi\circ \widetilde{k_*},D)$ and
$$\widetilde{\PD_{U_\vep}'}\bra{[\widetilde{k_*}]\grotimes[D]}
=\bbra{(E,\pi\circ \widetilde{k_*}\circ \beta_x,D)}_{x\in U_\vep}.$$
On the other hand, 
$\widetilde{\PD_{LM}}([D])
=\bbra{(E,\pi\circ \beta_x,D)}_{x\in LM}$ and hence
$$\widetilde{k^*}\bra{\widetilde{\PD_{LM}}([D])}
=\bbra{(E,\pi\circ \beta_x,D)}_{x\in U_\vep}.$$
Since $\widetilde{k_*}(\beta_x(f))=\beta_x(f)$, the square $(1)$ commutes.


$(2)$ We can prove it in a similar way of $(1)$. 
The essential point is the definition of $\widetilde{\tau_{U_{2\vep}}}$, that is to say, for $x\in M$ and $f\in \ca{S}_\vep$, we have $\widetilde{\tau_{U_{2\vep}}}(\beta_x^M(f))=\beta_x^{LM}(f)$, where $\beta_x^M$ is the Bott homomorphism for $M$ and $\beta_x^{LM}$ is that for $LM$.


$(3)$, $(4)$ They have been proved in Theorem \ref{non-cpt index thm}.
\end{proof}

\begin{cor}
For $[D]\in KK_{S^1}(\ca{A}(LM),\ca{S}_\vep)$, 
$$\widetilde{\ind_{S^1}^\pos}([D])=
\widetilde{t-\ind_{S^1}^\pos}(\widetilde{\PD_{LM}}([D])).$$
In particular, the localized index of $LM$ has a fixed-point formula.
\end{cor}

Although we have not constructed a concrete example of elements of $KK_{S^1}(\ca{A}(LM),\ca{S}_\vep)$, we can still compute the topological localized index of an element of $
\ca{R}KK_{S^1}(LM;\ca{S}_\vep\grotimes \scr{C}(LM),\ca{S}_\vep\grotimes \scr{C}(LM))$.

\begin{ex}
Let us compute the topological localized index homomorphism for $LS^2$.
Note that the following diagram commutes.
$$\begin{CD}
\ca{R}KK_{S^1}(LM;\scr{C}(LM), \scr{C}(LM))
@>\sigma_{\ca{S}_\vep}>>
\ca{R}KK_{S^1}(LM;\ca{S}_\vep\grotimes \scr{C}(LM),\ca{S}_\vep\grotimes \scr{C}(LM)) \\
@Vk^*VV @VV\widetilde{k^*}V \\
\ca{R}KK_{S^1}(U_\vep;\scr{C}(U_\vep),\scr{C}(U_\vep))
@>\sigma_{\ca{S}_\vep}>>
\ca{R}KK_{S^1}(U_\vep;\ca{S}_\vep\grotimes \scr{C}(U_\vep),\ca{S}_\vep\grotimes \scr{C}(U_\vep)) \\
@Vi^*VV @VV\widetilde{i^*}V \\
\ca{R}KK_{S^1}(M;{C}(M),{C}(M))
@>\sigma_{\ca{S}_\vep}>>
\ca{R}KK_{S^1}(M;\ca{S}_\vep\grotimes {C}(M),\ca{S}_\vep\grotimes {C}(M)) \\
@V[e^{-1}_{U_{2\vep}}]\grotimes-VV @VV[\widetilde{e^{-1}_{U_{2\vep}}}]\grotimes-V \\
\ca{R}KK_{S^1}(M;{C}(M),{C}(M))_\pos
@>\sigma_{\ca{S}_\vep}>>
\ca{R}KK_{S^1}(M;\ca{S}_\vep\grotimes {C}(M),\ca{S}_\vep\grotimes {C}(M))_\pos \\
@V{t-\ind_{M}'}VV @VV\widetilde{t-\ind_{M}}V \\
R(S^1)_\pos
@>\sigma_{\ca{S}_\vep}>>KK_{S^1}(\ca{S}_\vep,\ca{S}_\vep).
\end{CD}$$
We compute the composition of the left vertical arrows.
We denote it by $\bra{t-\ind_{S^1}^\pos}'$.
It is obvious that $\widetilde{t-\ind_{S^1}^\pos}(\sigma_{\ca{S}_\vep}(u))=\sigma_{\ca{S}_\vep}(({t-\ind_{S^1}^\pos})'(u))$ for $u\in \ca{R}KK_{S^1}(LM,\scr{C}(LM),\scr{C}(LM))$.

Let $\ca{E}$ be an $S^1$-equivariant vector bundle over $LM$ of finite rank. By definition, the topological localized index of it is given by
\begin{align*}
\int_{S^2}ch\bra{[e^{-1}_{U_{2\vep}}]\otimes \ca{E}|_{S^2}}\td(TS^2)
&=\int_{S^2}ch\bra{[e^{-1}_{U_{2\vep}}]}ch\bra{ \ca{E}|_{S^2}}\td(TS^2).
\end{align*}

Let us compute $\td(TS^2)$.
$TS^2$ has a $Spin^c$-structure indued by a complex structure.
Its first Chern class is given by $c_1(TS^2)=e(TS^2)=2x$, where $x$ is a generator of $H^2(S^2)$ so that $\int_{S^2}x=1$.
Thus, the total Todd class is $\td(TS^2)=
1+\frac{1}{2}c_1(TS^2)=1+x$.

Let us compute $ch([e^{-1}_{U_{2\vep}}])$.
The normal bundle $\nu_{U_{2\vep}\to M}$ is given by
$$\bigoplus_{n>0} T\otimes \bb{C}_n,$$
where $T=TS^2\otimes \bb{C}$.
The corresponding $K$-theory element can be written as $\sum_{n>0}T\otimes  q^n$ under the isomorphism
$\ca{R}KK_{S^1}(M;{C}(M),{C}(M))_\pos
\cong \ca{R}KK(M;{C}(M),{C}(M))\otimes R(S^1)_\pos$.

Since $T$ is orientable, $\wedge^2 T$ is trivial.
Thus, the exterior product is given by
$$\prod_{n>0}(1\otimes  q^0-
T\otimes  q^n+1\otimes  q^{2n}).$$
Its inverse is given by
$$\prod_{n>0}\bra{\sum_{l\geq 0}\bra{
T\otimes  q^n-1\otimes  q^{2n}}^l}.$$ 

We compute the Chern character of it.
Since $T$ is stably trivial, $ch(T)=2$.
Thus, 
the Chern character of the inverse Euler class is
\begin{align*}
\prod_{n>0}\bra{\sum_{l\geq 0}\bra{
2 q^n- q^{2n}}^l}
&=\prod_{n>0}\bra{
\frac{1}{1-2 q^n+ q^{2n}}} 
=\bra{\prod_{n>0}
\frac{1}{1- q^n}}^2 \\
&=\bra{\prod_{n>0}\sum_{l\geq 0} q^{ln}}^2 
=\bra{\sum_{n\geq 0}p(n) q^n}^2 \\
&=\sum_{n\geq 0}\sum_{m=0}^np(m)p(n-m) q^n
,\end{align*}
where $p(n)$ is the partition function.

Under the isomorphism $K^0_{S^1}(S^2)\cong K^0(S^2)\otimes R\bra{S^1}$, $\ca{E}$ can be written as $\sum_aE_a\otimes q^a$.
Therefore, the index is given by
$$\sum_a\bra{\int_{S^2}ch(E_a)(1+x)} \sum_{n\geq 0}\sum_{m=0}^np(m)p(n-m) q^{n+a}.$$ 
In particular, if $\ca{E}$ is the rank one trivial bundle, the index is given by
$$\sum_{n\geq 0}\sum_{m=0}^np(m)p(n-m) q^n.$$
\end{ex}

\begin{ex}
More generally, we can compute the loop space index for a compact Riemann surface. Let $\Sigma$ be the Riemann surface with genus $g$.
Since the tangent bundle of $\Sigma$ is stably trivial, the inverse Euler class is $\sum_{n\geq 0}\sum_{m=0}^np(m)p(n-m) q^n$.
The total Todd class is given by
$$1+\frac{1}{2}c_1(T\Sigma)=1+\frac{1}{2}(2-2g)x=1+(1-g)x,$$
where $x\in H^2(\Sigma)$ is the generator so that $\int_{\Sigma}x=1$.
Thus, the topological localized index of the trivial bundle of $L\Sigma$ is given by
$$(1-g)\sum_{n\geq 0}\sum_{m=0}^np(m)p(n-m) q^n.$$

\end{ex}

\subsection{Remained problems}

On the theme of the present paper, many problems are remained.

Since our construction is inspired by \cite{Wit}, the following is the most important.

\begin{prob}
Prove that the Witten genus is in some sense realized as the image of $\widetilde{\ind_{S^1}^\pos}$ of an appropriate element of $KK_{S^1}(\ca{A}(LM),\ca{S}_\vep)$.
\end{prob}

When we defined the localized index for loop spaces, we ignored the difference line bundle $[L]$.
The difference line bundle $L$ should be the ``difference'' between the ``Spinor bundle of $LM$'' and the tensor product of that of $M$ and the exterior algebra of $\nu_\bb{C}(M)$.
In \cite{Wit}, the $\zeta$-function renormalization appears.
Possibly these are related to each other.

\begin{prob}
$(1)$ Formulate an appropriate $L$.

$(2)$ Are there any relationship between $L$ and the $\zeta$-function renormalization in \cite{Wit}?
\end{prob}

Regarding the Witten rigidity, it is interesting to study the localized index for loop spaces from the viewpoint of the global structure of $LM$.
Since our index is a homomorphism from an invariant of $LM$, $KK_{S^1}(\ca{A}(LM),\ca{S}_\vep)$, the following problem is worth studying.

\begin{prob}
Study properties of $\widetilde{\ind_{S^1}^\pos}$ for $LM$ from the viewpoint of topology of $LM$ or the viewpoint of operator algebra theory of $\ca{A}(LM)$.
\end{prob}

Regarding an ``analysis on loop space'', we need to solve the following problem. It is our next challenge.

\begin{prob}
Construct an explicit element of $KK_{S^1}(\ca{A}(LM),\ca{S}_\vep)$.
\end{prob}

Such an element is nothing but a ``Dirac operator on $LM$''.



\appendix

\section{Our index homomorphism is appropriate}

In order to show that our index homomorphism is appropriate, we prove that we obtain the classical analytic index when we apply the construction of Section \ref{Non-compact manifolds} to a compact manifold.

Let $M$ be a compact Riemannian manifold equipped with an isometric $S^1$-action, $M^{S^1}$ the fixed-point set, $\nu(M^{S^1})$ the normal bundle,
$\nu(M^{S^1})_\delta:=\bbra{v\in \nu(M^{S^1}) \midd \|v\|<\delta}$, and $U_\delta:=\exp^\perp(\nu(M^{S^1})_\delta)$.
The projection of the normal bundle is denoted by $\varpi:\nu(M^{S^1})\to M^{S^1}$.
The situation is summarized as follows:
$$\xymatrix{
M^{S^1}  \ar@{^{(}->}^-{j}[r] \ar@/^18pt/^{i}[rr] & U_\delta \ar@{^{(}->}^-{k}[r] & M. \\
& \nu(M^{S^1} )_\delta \ar@{->>}^-{\varpi}[lu] \ar_-{\exp^{\perp},\cong}[u] &}$$

Then, since the following diagram commutes:
$$\xymatrix{
C(M) \ar^{i^*}[rr] && C(M^{S^1}) \\
& \bb{C} \ar[lu] \ar[ru]&}$$
we have the induced commutative diagram on $K$-homology
$$
\xymatrix{
K^{S^1}_0(M) \ar_-{\ind_{S^1}}[rd] &&
K^{S^1}_0(M^{S^1}) \ar_-{[i^*]\grotimes-}[ll] \ar^-{\ind_{S^1}}[ld] \\
&R(S^1)}$$
Thus, it suffices to find the ``inverse'' of $[i^*]$.

Let us observe $i^*$.
We have a proper inclusion $j:M^{S^1}\hookrightarrow U_\delta$ and an open inclusion $k:U_\delta \hookrightarrow M$.
Thus, we have a commutative diagram
$$\xymatrix{
& C_0(U_\delta) \ar_{k_*}[ld] \ar^{j^*}[rd] & \\
C(M) \ar_{i^*}[rr] && C(M^{S^1}) .}$$
It induces the following commutative diagram:
$$\xymatrix{
& KK_{S^1}(C_0(U_\delta),\bb{C}) & \\
KK_{S^1}(C(M),\bb{C})
 \ar^{[k_*]\grotimes-}[ru]   && 
KK_{S^1}(C(M^{S^1}),\bb{C}).
 \ar_{[j^*]\grotimes-}[lu] 
 \ar^{[i^*]\grotimes-}[ll] }$$

The correspondence $[k_*]\grotimes-$ is isomorphic.

\begin{lem}\label{Loc thm fps}
$k^*$ is isomorphic after tensoring with $R(S^1)_\pos$.
\end{lem}
\begin{proof}
%
Let $K:=M\setminus U_\delta$. It is a compact set equipped with an $S^1$-action.
By a parallel argument of the proof of \cite[Theorem 2.3]{HW}, it suffices to check that $K_0^{S^1}(K)_\pos=0$.
For this aim, since it is a unital module over $KK_{S^1}(C(K),C(K))_\pos$, we will prove that $KK_{S^1}(C(K),C(K))_\pos=0$.
For this aim, we will prove that there exists an invertible element in $R(S^1)_\pos$ so that the corresponding element in $KK_{S^1}(C(K),C(K))_\pos$ is zero.
The following construction is almost the same with \cite[Lemma 4.39]{Fur}.
See also \cite{Seg}

Thanks to the slice theorem, for arbitrary $x\in K$, there exists a submanifold $S_x$ through $x$ and an open neighborhood $U_x$ so that 
$S_x\times_{S^1_x}S^1\xrightarrow{\cong} U_x$, where the diffeomorphism is given by $[(s,g)]\mapsto g\cdot s$ and $S^1_x$ is the stabilizer of $x$.
Then, we can define a projection 
$$\pi_x: U_x\ni [(s,g)]\mapsto g S^1_x\in S^1/S^1_x.$$ 
Since $x$ is not a fixed-point, $S_x^1$ is a proper subgroup of $S^1$ and $S_x^1$ is closed, $\# S^1_x$ is finite.
Let $n_x=\# S^1_x$. Then $\pi_x$ is an $n_x$-fold covering on each orbit.

We prove that $U_x\times \bb{C}_0$ is $S^1$-equivariantly isomorphic to $U_x\times \bb{C}_{ln_x}$ for arbitrary $l\in \bb{Z}$.
We define an odd bundle automorphism 
$h_x:S^1\times (\bb{C}_0\groplus \bb{C}_{ln_x})\to S^1\times (\bb{C}_0\groplus \bb{C}_{ln_x})$ by
$$h_x(y)=\begin{pmatrix}
0 & \pi_x(y)^{-l} \\ \pi_x(y)^l & 0 \end{pmatrix}.$$
One can check that it is $S^1$-equivariant by a direct computation.
We denote $S^1\times (\bb{C}_0\groplus \bb{C}_{ln_x})$ by $E_x$.
This bundle extends to $K$ since it is topologically trivial, although $h_x$ does not.

From now on, we assume $l>0$. 
We can do the same construction for other points in $K$. Since $K$ is compact, there exists a finite set $\{x_1,x_2,\cdots,x_N\}$ so that $\cup_{i=1}^N U_{x_i}=K$. Let $\{\rho_i ^2\}$ be a partition of unity subordinate to $\{U_{x_i}\}$.
Let $E:=\grotimes_{i=1}^NE_{x_i}$ and let $h$ be the odd automorphism defined by
$$h:=\sum_{i=1}^N\epsilon_{1}\grotimes \cdots \grotimes \epsilon_{i-1}
\grotimes \rho_ih_{x_i}\grotimes \id\grotimes \cdots\grotimes \id,$$
where $\vep_1$'s are the grading homomorphism 
$\vep_1=\begin{pmatrix}
1 & 0 \\ 0 & -1 \end{pmatrix}$, and $l_i$ to define $h_{x_i}$ is chosen to be positive.
Thanks to $\rho_i$, it is well-defined.
Since 
$$h^2=\sum \id\grotimes \cdots\grotimes \id\grotimes \rho_i^2h_{x_i}^2\grotimes \id\grotimes \cdots\grotimes \id,$$
the homomorphism $h$ is fiberwisely isomorphic. Therefore, $[(E,h)]=0$ as an element of $K_{S^1}^0(K)$.

On the other hand, since the corresponding element in $R(S^1)_\pos$ is $E=\prod_{i=1}^N( q^0\groplus q^{l_in_{x_i}})=
q^0+(\text{higher terms})$, it is invertible in $R(S^1)_\pos$. 
Since an invertible element vanishes, the ring $KK_{S^1}(C(K),C(K))_\pos$ is trivial.
\end{proof}
\begin{rmk}
This construction works if $K$ is a Lindel\"of space. In order to define $E$, choose $l_i=i$.
Then, the infinite-rank vector bundle $E$ makes sense in $\ca{R}KK{S^1}(K;C_0(K),C_0(K))_\pos$, and it is invertible.
We do not study this generalization in the present paper.\end{rmk}

Thus, the inverse of $[i^*]$ is the composition of $[k_*]$ and the inverse of $[j^*]\in KK_{S^1}(C_0(U_\delta),C(M^{S^1}))$.
Let $[\tau_{U_\delta}]\in KK_{S^1}(C(M^{S^1}),C_0(U_\delta))$ be the Thom class.
In order to find out the inverse of $[j^*]$, we consider the following commutative diagram.
$$\xymatrix{
KK_{S^1}(C_0(U_\delta),\bb{C}) 
\ar^{[\tau_{U_\delta}]\grotimes-}[rd] \\
KK_{S^1}(C(M^{S^1}),\bb{C}) \ar^{[j^*]\grotimes-}[u]  \ar_{[e_{U_\delta}]\grotimes-}[r]& 
KK_{S^1}(C(M^{S^1}),\bb{C}).}$$

Thus, the inverse of $[j^*]$ is the composition of $[\tau_{U_\delta}]$ and the inverse of the Euler class.
Consequently, we have the following result.

\begin{thm}\label{cpt ana index}
The following diagram commutes:
$$\xymatrix{
KK_{S^1}(C(M),\bb{C})
 \ar_-{[\ud{\bb{C}}_M]\grotimes-}[d]
 \ar^-{[k_*]\grotimes-}[r] &
KK_{S^1}(C_0(U_\delta),\bb{C})
 \ar^{[\tau_{U_\delta}]\grotimes-}[r] &
KK_{S^1}(C(M^{S^1}),\bb{C})
 \ar^{[e^{-1}_{U_\delta}]\grotimes-}[r] &
KK_{S^1}(C(M^{S^1}),\bb{C})_\pos
 \ar^-{[\ud{\bb{C}}_{M^{S^1}}]\grotimes-}[d] \\
R(S^1) 
 \ar_-{\pos}[rrr] &&& 
R(S^1)_\pos.}$$
\end{thm}

The fixed-point formula is obtained in the same way of Section \ref{Fixed-point formula of the localized index}.

\section*{Acknowledgements}
I am supported by JSPS KAKENHI Grant Number 21K20320.

Doman Takata, 
Faculty of Education Mathematical and Natural Sciences,
Niigata University, 
8050 Ikarashi 2-no-cho, Nishi-ku, Niigata, 950-2181, Japan. 

E-mail address: {\tt d.takata@ed.niigata-u.ac.jp}

\end{document}